\documentclass{article}

\usepackage{latexsym,url}
\usepackage{geometry}                
\usepackage{graphicx}
\usepackage{amssymb}
\usepackage{pgf}
\usepackage{amsmath, amsfonts}
\usepackage{tikz}
\usepackage{graphicx}
\usepackage{graphicx}
\usepackage{epstopdf}
\usepackage{fix-cm}
\usepackage{multirow}
\usepackage[english]{babel}
\usepackage{dcolumn}
\usepackage{color, cancel}
\usepackage[normalem]{ulem}
\usepackage{caption, subcaption}
\usepackage{authblk}
\usepackage{fancyhdr}

\newtheorem{theorem}{Theorem}[section]

\newtheorem{definition}[theorem]{Definition}

\newtheorem{lemma}[theorem]{Lemma}

\newtheorem{example}[theorem]{Example}

\numberwithin{equation}{section}

\def \bn {\hfill \\ \smallskip\noindent}

\def\proof{\bn {\bf Proof.} }

\newcommand{\tr}{{\rm Tr\, }}

\newcommand{\qed}{\hspace*{\fill} $\Box$ \\ \medskip}

\newcommand{\dJN}{\delta^{(r)}_{\KPM}}
\newcommand{\dJr}{\delta^{(r)}_{\KPM}}

\newcommand{\oN}{{\mathbb N}}
\newcommand{\oR}{{\mathbb R}}

\newcommand{\Supp}{\text{\rm Supp}}
\newcommand{\bK}{{\mathbf K}}
\newcommand{\fmin}{f_{\min}}
\newcommand{\KPM}{\text{KPM}}
\newcommand{\ignore}[1]{}

\begin{document}
\thispagestyle{empty}

\title{Improved convergence rates for Lasserre-type hierarchies of upper bounds for box-constrained polynomial optimization}
\author{Etienne de Klerk \\
              Tilburg University \\
              PO Box 90153, 5000 LE Tilburg, The Netherlands \\
              \texttt{E.deKlerk@uvt.nl}
           \and
           Roxana Hess\footnote{Most of this work was done while the second author was staying at CWI in autumn 2015. She would like to thank CWI and, in particular M. Laurent for the hospitality and support during her stay, and Universit\'e Paul Sabatier, \'Ecole Doctorale Syst\`emes and \'Ecole des Docteurs de l'Universit\'e F\'ed\'ederale Toulouse Midi-Pyr\'en\'ees for the funding.}\\ 
     							LAAS-CNRS, Universit\'e de Toulouse\\
							LAAS, 7 avenue du colonel Roche, 31400 Toulouse, France\\
							\texttt{rhess@laas.fr}
           \and
           Monique Laurent \\
              Centrum Wiskunde \& Informatica (CWI), Amsterdam and Tilburg University \\
              CWI, Postbus 94079, 1090 GB Amsterdam, The Netherlands\\
              \texttt{M.Laurent@cwi.nl}
}
\renewcommand\Affilfont{\small}
\setlength{\affilsep}{0.2cm}


\ignore{
\institute{Etienne de Klerk \at
              Tilburg University \at
              PO Box 90153, 5000 LE Tilburg, The Netherlands \\
              \email{E.deKlerk@uvt.nl}
              \and
              Roaxana Hess \at ??
           \and
           Monique Laurent \at
              Centrum Wiskunde \& Informatica (CWI), Amsterdam and Tilburg University \at
              CWI, Postbus 94079, 1090 GB Amsterdam, The Netherlands\\
              \email{M.Laurent@cwi.nl}
}
}

\maketitle

\begin{abstract}
We consider the problem of minimizing a given $n$-variate polynomial $f$ over the hypercube $[-1,1]^n$.
An idea introduced by Lasserre, is to find a probability distribution on $[-1,1]^n$ with polynomial density function $h$ (of given degree $r$) that
minimizes the expectation $\int_{[-1,1]^n} f(x)h(x)d\mu(x)$, where $d\mu(x)$ is a fixed, finite Borel measure supported on $[-1,1]^n$.
It is known that, for the Lebesgue measure $d\mu(x) = dx$, one may show an error bound $O(1/\sqrt{r})$ if $h$ is  a sum-of-squares density, and
an $O(1/r)$ error bound if $h$ is the density of a beta distribution.
In this paper, we show an error bound of $O(1/r^2)$, if $d\mu(x) = \left(\prod_{i=1}^n \sqrt{1-x_i^2} \right)^{-1}dx$ (the well-known measure in the
study of orthogonal polynomials), and $h$ has a Schm\"udgen-type representation with respect to $[-1,1]^n$, which is a more general condition than
a sum of squares.
The convergence rate analysis relies on the theory of polynomial kernels and, in particular, on Jackson kernels.
We also show that the resulting upper bounds may be computed as generalized eigenvalue problems, as is also the case for sum-of-squares densities.
\end{abstract}



\section{Introduction}

\subsection{Background results}
We consider the problem of minimizing a given $n$-variate polynomial $f \in \mathbb{R}[x]$
over the compact set $\bK=[-1,1]^n$, i.e., computing the parameter
\begin{eqnarray}\label{eqfmin}
f_{\min} &=& \min_{x \in \bK} f(x).
\end{eqnarray}
This is a hard optimization problem which contains, e.g.,   the well-known NP-hard maximum stable set and maximum cut problems  in graphs
 (see, e.g., \cite{PH11,PH13}).
It  falls within box-constrained  (aka  bound-constrained)  optimization which has been widely studied in the literature.
In particular  iterative
methods for bound-constrained optimization are described in the books \cite{Bertsekas1996,Fletcher_book,Gill_Murray_Wright}, including projected gradient and active set methods.
The
latest algorithmic developments for box-constrained global optimization are surveyed in
the recent thesis \cite{Palthesis}; see also \cite{Hager_Zhang} and the references therein for  recent work on active set methods, and a list of applications. 
The box-constrained optimization problem is even of practical interest in the (polynomially solvable) case where $f$ is a convex quadratic problem, 
and dedicated active set methods have been developed for this case; see \cite{Hungerlander_Rendl}.

In this paper we will focus on the question of finding a sequence of upper bounds converging to the global minimum and allowing a known estimate on the rate of convergence.
It should be emphasized that it is in general a difficult challenge in non-convex optimization to obtain such results.
Following Lasserre \cite{Las01,Las11}, our approach will be based on reformulating problem (\ref{eqfmin}) as an optimization problem over measures and then restricting it to   subclasses of measures that we are able to analyze. Sequences of upper bounds have been recently proposed and analyzed in  \cite{ConvAna,ElemCalc}; in the present paper we will propose new bounds for which we can prove a sharper  rate of convergence. We now  introduce our approach.

As observed by Lasserre \cite{Las01},  problem (\ref{eqfmin}) can be reformulated as
\begin{eqnarray*}
f_{\min}
&=& \min_{\mu \in \mathcal{M}(\bK)} \int_{\bK} f(x)d\mu(x),
\end{eqnarray*}
where $\mathcal{M}(\bK)$ denotes the set of probability measures supported on $\bK$.
Hence an upper bound on $f_{\min}$ may be obtained by considering a fixed probability measure $\mu$ on $\bK$.
In particular, the optimal value $f_{\min}$ is obtained when selecting for $\mu$  the Dirac measure at a global minimizer $x^*$ of $f$ in $\bK$.

Lasserre \cite{Las11} proposed the following strategy to build a hierarchy of upper bounds converging to  $\fmin$.  The idea is to do successive approximations of the Dirac measure at $x^*$ by using sum-of-squares (SOS) density functions of  growing degrees. More precisely, Lasserre \cite{Las11}  considered  a set of Borel measures $\mu_r$ obtained by selecting a fixed, finite Borel measure $\mu$ on $\bK$ (like, e.g., the Lebesgue measure) together with a polynomial density function that is a
sum-of-squares (SOS) polynomial of given degree $r$.

 When selecting for $\mu$ the Lebesgue measure on $\bK$ this leads to the following  hierarchy of
upper bounds on $f_{\min}$, indexed by $r\in \oN$: 
\begin{eqnarray}\label{fundr}
\underline{f}^{(r)}_{\mathbf{K}}:=\inf_{h\in\Sigma[x]_r}\int_{\mathbf{K}}h(x)f(x)dx \ \ \text{s.t. $\int_{\mathbf{K}}h(x)dx=1$,}
\end{eqnarray}
where  $\Sigma[x]_r$ denotes the set of  sum-of-squares polynomials  of degree at most $r$.

The convergence to $\fmin$ of the bounds $\underline f^{(r)}_{\bK}$ is an immediate consequence of the following theorem, which holds for
general compact sets $\bK$ and continuous functions $f$.

\begin{theorem}\label{thmlasnn}\cite[cf.\ Theorem 3.2]{Las11}
Let $\mathbf{K}\subseteq \oR^n$ be compact, let $\mu$ be an arbitrary finite Borel measure supported by $\mathbf{K}$, 
 and let $f$ be a continuous function on $\oR^n$. Then, $f$ is nonnegative on $\mathbf{K}$ if and only if
\begin{equation*}
\int_{\mathbf{K}}fg^2d\mu\ge0 \ \ \forall  g\in\oR[x].
\end{equation*}
Therefore, the minimum of $f$ over $\mathbf{K}$ can be expressed as
\begin{equation}\label{formulamu}
f_{\min}=\inf_{h \in \Sigma[x]}\int_{\mathbf{K}}fhd\mu \ \ \text{s.t. $\int_{\mathbf{K}}hd\mu=1$.}
\end{equation}
\end{theorem}

As already mentioned in \cite{ConvAna}, formula \eqref{formulamu} does not appear explicitly in \cite{Las11} which only  mentions the characterization  of nonnegative functions, but one can derive it easily from this nonnegativity characterization. To see this we write $f_{\min} = \sup\{\lambda: f(x)-\lambda\geq 0 \text{ on } \mathbf{K}\}$. Then, 
for any finite Borel measure $\mu$, we have $f_{\min}=\sup\{\lambda: \int_{\mathbf{K}}h(f-\lambda)d\mu\geq 0\ \forall h\in \Sigma[x]\}$. As $\int_{\mathbf{K}}h(f-\lambda)d\mu = \int_{\mathbf{K}}hf\, d\mu - \lambda\int_{\mathbf{K}}h\, d\mu$, after normalizing $\int_{\mathbf{K}}h\, d\mu = 1$, the formula (\ref{formulamu}) follows.

In the recent work \cite{ElemCalc}, it is shown that for a compact set $\mathbf{K}\subseteq  [0,1]^n$ one may obtain a similar result
using  density functions arising from  (products of univariate) beta distributions.
In particular, the following theorem is implicit in \cite{ElemCalc}.

\begin{theorem}\cite{ElemCalc}
Let $\mathbf{K}\subseteq  [0,1]^n$ be a compact set, let $\mu$ be an arbitrary finite Borel measure supported by $\mathbf{K}$, 
 and let $f$ be a continuous function on $\oR^n$. Then, $f$ is nonnegative on $\mathbf{K}$ if and only if
\begin{equation*}
\int_{\mathbf{K}}fhd\mu\ge0
\end{equation*}
for all $h$ of the form
\begin{equation}
\label{multivariate beta density}
h(x)= \frac{    \prod_{i=1}^n x_i^{\beta_i}(1-x_i)^{\eta_i}  }{ \int_{\bK}  \prod_{i=1}^n x_i^{\beta_i}(1-x_i)^{\eta_i} },
\end{equation}
where the $\beta_i'$s and $\eta_i'$s are nonnegative integers.
Therefore, the minimum of $f$ over $\mathbf{K}$ can be expressed as
\begin{equation}\label{formulamu1}
f_{\min}=\inf_{h}\int_{\mathbf{K}}fhd\mu \ \ \text{s.t. $\int_{\mathbf{K}}hd\mu=1$,}
\end{equation}
where the infimum is taken over all beta-densities $h$ of the form \eqref{multivariate beta density}.
\end{theorem}

For the box $\bK=[0,1]^n$ and  selecting for $\mu$ the Lebesgue measure, we obtain a hierarchy of upper bounds $f^H_r$ converging to $\fmin$, where $f^H_r$ is the optimum value of the program (\ref{formulamu1}) when the infimum is taken over all beta-densities $h$ of the form \eqref{multivariate beta density} with degree  $r$.

The rate of convergence  of the upper bounds $\underline{f}^{(r)}_{\mathbf{K}}$ and $f^H_r$ has been investigated recently in  \cite{ConvAna} and \cite{ElemCalc}, respectively.
It is shown in \cite{ConvAna} that $\underline{f}^{(r)}_{\mathbf{K}} -\fmin=O(1/\sqrt r)$ for a large class of compact sets $\bK$ (including all convex bodies and thus the box $[0,1]^n$ or $[-1,1]^n$) and the stronger rate $f^H_r-\fmin =O(1/r)$ is shown in \cite{ElemCalc} for the box $\bK=[0,1]^n$.
While the parameters $\underline f^{(r)}_{\bK}$ can be computed using semidefinite optimization (in fact, a generalized eigenvalue computation problem, see \cite{Las11}), an advantage of the parameters $f^H_r$ is that their computation involves only elementary operations (see \cite{ElemCalc}).

Another possibility for getting a hierarchy of upper bounds is grid search, where one takes the best function evaluation at all rational points in $\mathbf K=[0,1]^n$ with given denominator $r$. It has been shown in \cite{ElemCalc} that these bounds  have a rate of convergence in $O(1/r^2)$. However, the computation of the order $r$ bound needs  an exponential number $r^n$ of function evaluations.

\subsection{New contribution}
In the present work we continue this line of research. For the  box $\bK=[-1,1]^n$,  our objective is to build  a new hierarchy of measure-based upper bounds, for which we will be able to show a sharper rate of convergence in $O(1/r^2)$.
We obtain these upper bounds by considering a specific Borel measure $\mu$   (specified below in (\ref{eqmu})) and  polynomial density functions with a so-called Schm\"udgen-type SOS representation (as in (\ref{eqSch}) below).

We first recall the relevant result of Schm\"udgen \cite{Sch91}, which gives SOS representations for positive polynomials on a basic closed semi-algebraic set (see also, e.g., \cite{DP},\cite[Theorem 3.16]{sos}, \cite{Mar}).

\begin{theorem}[Schm\"udgen \cite{Sch91}]
\label{thm:Schmuedgen}
Consider the set   $\mathbf{K} = \{{x} \in \mathbb{R}^n\mid  g_1(x)\geq 0,\dotsc,g_m(x) \geq 0\}$, where $g_1,\dotsc,g_m \in\mathbb{R}[{x}]$,
	and assume that $\bK$ is  compact.
 If $p \in \mathbb{R}[x]$ is positive on $\bK$, then $p$ can be written as $p = \sum_{I \subseteq [m]} \sigma_I \prod_{i\in I} g_i$, where $\sigma_I$ ($I \subseteq [m]$) are sum-of-squares polynomials.
\end{theorem}

For the box $\bK = [-1,1]^n$, described by the polynomial inequalities $1-x_1^2\ge 0 ,\ldots,1-x_n^2\ge 0$, we consider polynomial densities that allow a Schm\"udgen-type representation of bounded degree $r$:
\begin{equation}\label{eqSch}
h(x) =  \sum_{I \subseteq [n]} \sigma_I(x) \prod_{i\in I} (1-x_i^2),
\end{equation}
where the polynomials $\sigma_I$ are sum-of-squares polynomials with degree at most $r-2|I|$ (to ensure  that the degree of $h$ is at most $r$).
We will also fix the following Borel measure $\mu$ on $[-1,1]^n$ (which, as will be recalled below, is associated with some orthogonal polynomials):
\begin{equation}\label{eqmu}
d\mu(x) = \left(\prod_{i=1}^n \pi \sqrt{1-x_i^2} \right)^{-1}dx.
\end{equation}
This leads to the following new hierarchy of upper bounds  ${f}^{(r)}$ for $ f_{\min}$. 

\begin{definition}\label{defnewbound}
Let $\mu$ be the Borel measure from (\ref{eqmu}). For $r\in \oN$ consider the  parameters
\begin{eqnarray}\label{fr}
{f}^{(r)}:=\inf_{h}\int_{[-1,1]^n}f h d\mu \ \ \text{s.t. $\int_{[-1,1]^n}hd\mu=1$,}
\end{eqnarray}
where the infimum is taken over the polynomial densities $h$ that allow a Schm\"udgen-type  representation (\ref{eqSch}), where each $\sigma_I$ is a sum-of-squares polynomial with degree at most $r-2|I|$.
\end{definition}

The convergence of the parameters $f^{(r)}$ to $\fmin$ follows as a direct application of Theorem \ref{thmlasnn}, since $f_{\min} \leq f^{(r+1)} \leq f^{(r)}$ for all $r$ and sums of squares allow a Schm\"udgen-type representation. As a small remark, note that due to the fact that $[-1,1]^n$ has a nonempty interior the program \eqref{fr} has an optimal solution $h^*$ for all $r$ by \cite[Theorem 4.2]{Las11}.

A  main result in this paper is to show that the bounds $f^{(r)}$ have a rate of convergence in $O(1/r^2)$.
Moreover we will show that the parameter $f^{(r)}$ can be computed through  generalized eigenvalue computations.

\begin{theorem}\label{theodJNn}
Let $f \in \oR[x]$ be a polynomial and $\fmin$ be its minimum value over the box $[-1,1]^n$. For any $r$ large enough, the parameters  $f^{(r)}$  defined in (\ref{fr}) satisfy
$$f^{(r)}-f_{\min} = O\left(\frac{1}{r^2}\right).$$
\end{theorem}

As already observed above  this result compares favorably with the estimate:  $\underline{f}^{(r)}_{\mathbf{K}} - f_{\min} = O\left(\frac{1}{\sqrt{r}}\right)$ shown in  \cite{ConvAna} for the bounds $f^{(r)}_{\bK}$ based on  using SOS densities.
(Note however that the latter convergence rate holds for a larger class of  sets $\mathbf{K}$ that includes all convex bodies; see \cite{ConvAna} for details.)
The new result also improves the estimate $f^H_r-\fmin=O\left(\frac{1}{{r}}\right)$, shown in \cite{ElemCalc} for the bounds $f^H_r$ obtained by using  densities arising from beta distributions.

We now illustrate the optimal densities appearing in the new bounds $f^{(r)}$ on an example.

\begin{example}
Consider the minimization of the Motzkin polynomial $$f(x_1,x_2)=64(x_1^4x_2^2+x_1^2x_2^4) - 48x_1^2x_2^2 +1$$ over the hypercube $[-1,1]^2$,
 which has four global minimizers at the points
$\left(\pm \frac{1}{2},\pm \frac{1}{2}\right)$, and $f_{\min} = 0$.
Figure \ref{figure:motzkin18} shows the
optimal density function $h^*$ computed when solving the problem  (\ref{fr}) for  degrees $12$ and $16$, respectively.
Note that the optimal density $h^*$ shows four peaks at the four global minimizers of $f$ in $[-1,1]^2$.
The corresponding upper bounds from (\ref{fr}) are $f^{(12)} = 0.8098$ and $f^{(16)} = 0.6949$.

\begin{figure}[h!]
\begin{center}
\includegraphics[width=0.4\textwidth]{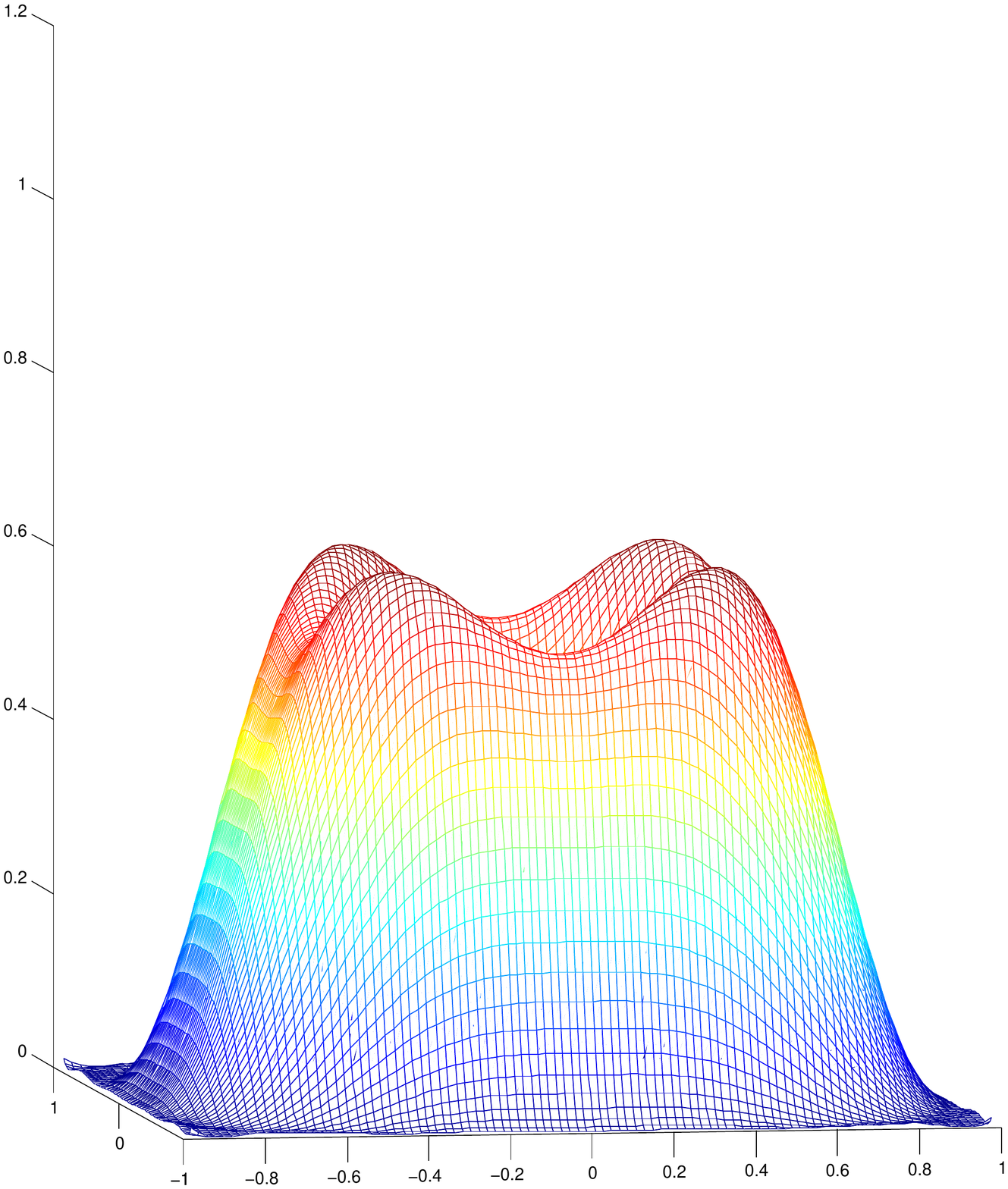}
\includegraphics[width=0.4\textwidth]{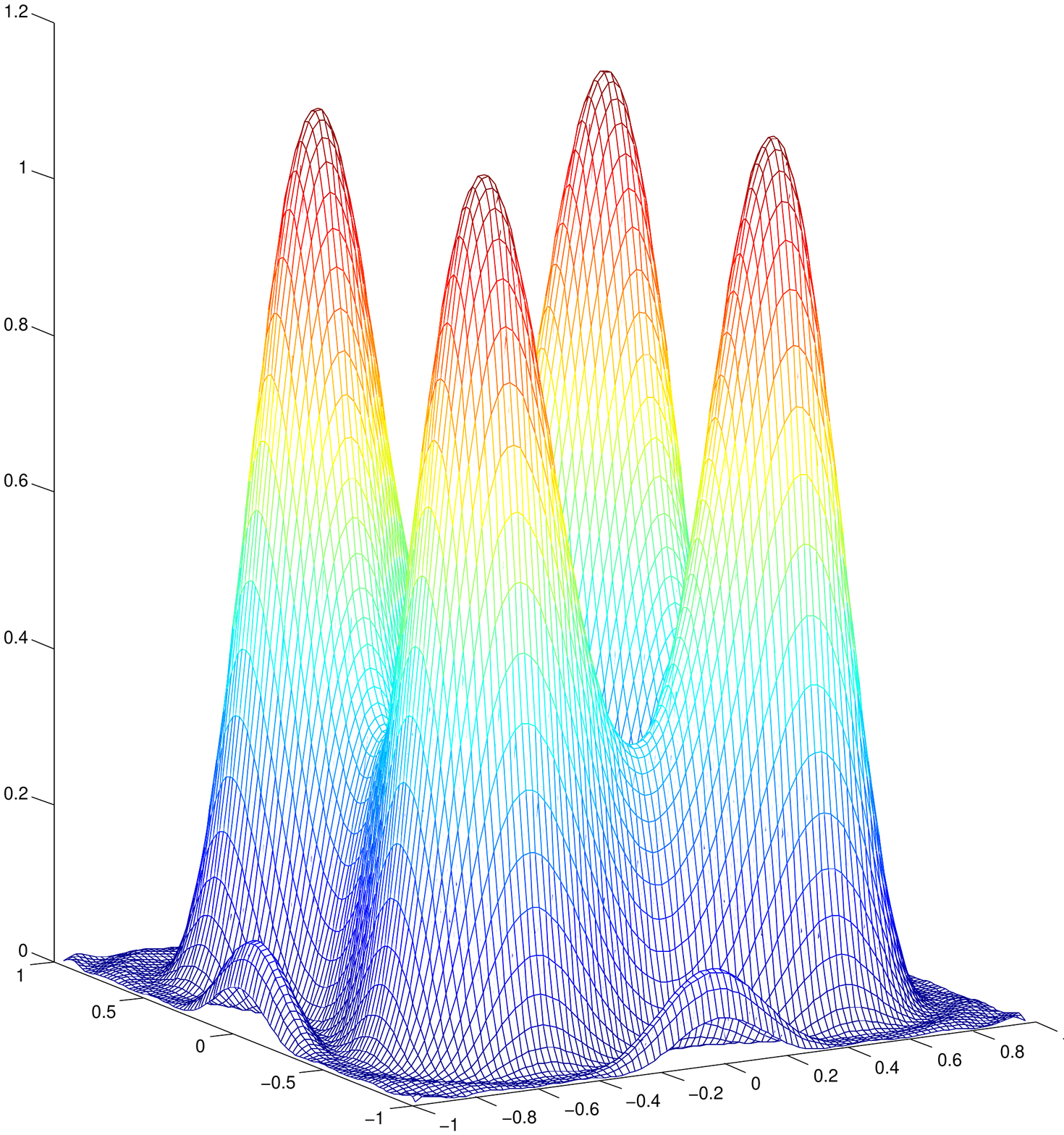}\\
  \caption{\label{figure:motzkin18}Graphs of $h^*$ on $[-1,1]^2$ ($\deg(h^*)=12,16$) for the Motzkin polynomial.}
\end{center}
\end{figure}

\end{example}




\subsubsection*{Strategy and outline of the paper}

In order to show the convergence rate in $O(1/r^2)$ of Theorem \ref{theodJNn} we need to exhibit a polynomial density function $h_r$ of degree at most $r$ which admits an SOS representation of Schm\"udgen-type and for which we are able to show that
$\int_{[-1,1]^n} fhd\mu -\fmin =O(1/r^2).$
The idea is to find such a polynomial density which approximates well the Dirac delta function at a global minimizer $x^*$ of $f$ over $[-1,1]^n$.
For this we will use  the well-established polynomial kernel method (KPM) and, more specifically, we will use the Jackson kernel, a well known tool in approximation theory  to yield best (uniform) polynomial approximations of continuous functions.



The paper is organized as follows.
 Section \ref{secbackground} contains  some background information   about the polynomial kernel method needed for our analysis of the new bounds
 $f^{(r)}$. Specifically, we introduce Chebyshev polynomials in Section \ref{secChebyshev} and Jackson kernels in Section \ref{secJacksonkernel}, and   then we use them in Section \ref{secJacksonapprox} to construct  suitable polynomial densities  $h_r$ giving good approximations of  the Dirac delta function at a global minimizer of $f$ in the box.
 We then carry out the analysis of the upper bounds on $\fmin$  in Section \ref{secunivariate} for the univariate case and in Section \ref{secmultivariate} for the general multivariate case, thus proving the result of Theorem \ref{theodJNn}.
 In Section \ref{seceigenvalue} we show how the new  bounds $f^{(r)}$ can be computed as generalized eigenvalue problems and in  Section \ref{secnumerical} we conclude with some numerical examples illustrating the behavior of the  bounds $f^{(r)}$.

\subsubsection*{Notation}

Throughout, $\Sigma[x]$ denotes the set of all sum-of-squares (SOS) polynomials (i.e., all polynomials $h$ of the form
$h=\sum_{i=1}^k p_i(x)^2$ for some polynomials $p_1,\ldots,p_k$  and  $k\in \oN$) and
 $\Sigma[x]_r$ denotes the set of  SOS polynomials  of degree at most $r$ (of the form $h=\sum_{i=1}^k p_i(x)^2$ for some polynomials $p_i$ of degree at most $r/2$).
For $\alpha\in \oN^n$,  $\Supp(\alpha)=\{i\in [n]: \alpha_i\ne 0\}$ denotes the support of $\alpha$ and, for $\alpha,\beta\in \oN^n$, $\delta_{\alpha,\beta}\in \{0,1\}$ is equal to 1 if and only if $\alpha=\beta$.

 \section{Background on the polynomial kernel method}\label{secbackground}

Our goal is to approximate the Dirac delta function at a given point $x^*\in \oR^n$ as well as possible, using  polynomial density functions  of bounded  degrees.
This is  a classical question in approximation theory.
In this section we will review how this may be done using the polynomial kernel method and, in particular, using Jackson kernels.
This theory is usually developed using the Chebyshev polynomials, and we start by reviewing their properties.
We will follow mainly the work \cite{KPM} for our exposition and we  refer to the handbook \cite{AS}  for more background information.

\subsection{Chebyshev polynomials }\label{secChebyshev}

We will use the  univariate polynomials $T_k(x)$ and $U_k(x)$, respectively known as the Chebyshev polynomials of the first and second kind.
They are  defined as follows:
\begin{equation}\label{eqTUn}
T_k(x)= \cos(k\arccos (x)),\ \ \ U_k(x)=\frac{\sin((k+1)\arccos(x))}{ \sin (\arccos (x))} \ \ \text{ for } \ x\in [-1,1],\ k \in \oN,
\end{equation}
and they satisfy the following recurrence relationships:
\begin{equation}\label{eqTnrec}
T_0(x)=1,\ T_{-1}(x)=T_1(x)=x, \ T_{k+1}(x)= 2xT_k(x)-T_{k-1}(x),
\end{equation}
\begin{equation}\label{eqUnrec}
U_0(x)=1, \ U_{-1}(x)=0,  
\ U_{k+1}(x)=2xU_k(x)-U_{k-1}(x).
\end{equation}
As a direct application one can verify that
\begin{equation}\label{eqval1}
\begin{split}
T_k(0) &= \begin{cases}
	0 & \text{ for }k\text{ odd},\\
	(-1)^{\frac{k}{2}} & \text{ for }k\text{ even},
	\end{cases}\\
	T_k(1)&=1, \ U_k(1)=k+1,\ U_k(-1)=(-1)^k(k+1) \ \ \ \text{ for } k\in \oN.
	\end{split}
\end{equation}
The Chebyshev polynomials have the extrema
\[
\max_{x \in [-1,1]} |T_k(x) | = 1 \mbox{ and } \max_{x \in [-1,1]} |U_k(x) | = k+1,
\]
attained at $x = \pm 1$ (see, e.g., \cite[\S22.14.4, 22.14.6]{AS}).

The Chebyshev polynomials  are orthogonal for the following inner product on the space of integrable functions over $[-1,1]$:
\begin{equation}\label{eqinner}
\langle f,g\rangle =\int_{-1}^1 \frac{f(x)g(x)}{\pi \sqrt{1-x^2}} dx,
\end{equation}
and their orthogonality relationships read
\begin{equation}\label{eqortho}
\langle T_k,T_m \rangle = 0 \text{ if } k\ne m,\
\langle T_0,T_0 \rangle =1, \
\langle T_k,T_k\rangle = \tfrac{1}{2} \ \text{ if } k\ge 1.
\end{equation}
For any $r\in \oN$ the Chebyshev polynomials $T_k$ ($k\le r$) form a basis of the space of univariate polynomials with degree at most $r$.
One may write the Chebyshev polynomials in the standard monomial basis using the relations
\begin{align*}
	T_k(x) &= \sum_{i=0}^k t_i^{(k)}x^i = \frac{k}{2}\sum_{m=0}^{\lfloor \frac{k}{2}\rfloor} (-1)^m\frac{(k-m-1)!}{m!(k-2m)!}(2x)^{k-2m}, &k>0\\
	U_{k-1}(x) &= \sum_{i=0}^{k-1}u_i^{(k)}x^i = \sum_{m=0}^{\lfloor \frac{k-1}{2}\rfloor} (-1)^m\frac{(k-m-1)!}{m!(k-1-2m)!}(2x)^{k-1-2m}, &k>1;
\end{align*}
see, e.g., \cite[Chap. 22]{AS}.
From this, one may derive a bound on the largest coefficient in absolute value appearing in the above expansions of $T_k(x)$ and $U_{k-1}(x)$.
A proof for the  following result will be given in the appendix.

\begin{lemma}\label{lemChebyshev}
For any fixed integer $k > 1$,
one has
\begin{equation}
\label{Cheb coeff bound}
\max_{0\le i \le k-1} |u_i^{(k)}| \le \max_{0\le i \le k} |t_i^{(k)}| = 2^{k-1-2\psi(k)}\frac{k(k-\psi(k)-1)!}{\psi(k)!(k-2\psi(k))!},
\end{equation}
where  $\psi(k) = 0$ for $k\leq 4$ and $\psi(k)=\left\lceil\frac{1}{8}\left(4k-5 - \sqrt{8k^2-7}\right)\right\rceil$ for $k\geq 4$. Moreover, the right-hand side of (\ref{Cheb coeff bound}) increases monotonically with increasing $k$.
\end{lemma}
In the multivariate case we use the following notation. We let $d\mu(x)$ denote the Lebesgue measure on $[-1,1]^n$ with the function
$\prod_{i=1}^n \left(\pi\sqrt{1-x_i^2}\right)^{-1}$ as density function:
\begin{equation}\label{eqmuC}
d\mu(x)= \prod_{i=1}^n \left(\pi\sqrt{1-x_i^2}\right)^{-1} dx
\end{equation}
and we consider the following inner product for two integrable functions $f,g$ on the box $[-1,1]^n$:
$$\langle f,g\rangle =\int_{[-1,1]^n}  f(x)g(x) d\mu(x)$$
(which coincides with (\ref{eqinner}) in the univariate case $n=1$).
For $\alpha\in \oN^n$, we define the multivariate Chebyshev polynomial
$$T_\alpha(x)=\prod_{i=1}^n T_{\alpha_i}(x_i) \ \text{ for } x\in \oR^n.$$
The multivariate Chebyshev polynomials satisfy the following orthogonality relationships:
\begin{equation}\label{eqorthomulti}
\langle T_\alpha,T_\beta\rangle = \left(\frac{1}{2}\right)^{|\Supp(\alpha)|}  \delta_{\alpha,\beta}
\end{equation}
and, for any $r\in \oN$,  the set of Chebyshev polynomials $\{T_\alpha(x): |\alpha|\le r\}$ is a basis of the space of $n$-variate polynomials of degree at most $r$.

\subsection{Jackson kernels}\label{secJacksonkernel}

A classical problem in approximation theory is to find a best (uniform)  approximation of a given continuous function $f:[-1,1]\rightarrow \oR$   by a polynomial of given maximum degree $r$.
Following \cite{KPM}, a possible approach is to take the convolution $f_\KPM^{(r)}$ of $f$ with a  kernel function of the form
$$K_r(x,y)= \frac{1}{\pi \sqrt{1-x^2}\pi \sqrt{1-y^2}} \left(g^r_0T_0(x)T_0(y)+2\sum_{k=1}^r g_k^r T_k(x)T_k(y)\right),$$ where $r\in \oN$ and
the coefficients $g_k^r$ are selected so that the following properties hold:
\begin{itemize}
\item[(1)] The kernel is positive: $K_r(x,y)>0$ for all $x,y\in [-1,1]$.
\item[(2)] The kernel is normalized: $g^r_0=1$.
\item[(3)]  The second  coefficients $g^r_1$ tend to 1 as $r\rightarrow \infty$.
\end{itemize}
The function $f_\KPM^{(r)}$ is then defined by
\begin{equation}\label{eqfKPM}
f_\KPM^{(r)}(x)=\int_{-1}^1 \pi \sqrt{1-y^2} K_r(x,y) f(y) dy.
\end{equation}
\ignore{
 We consider the approximation of a continuous function $f:[-1,1] \rightarrow \mathbb{R}$
given by a Hilbert-Schmidt integral operator of the type:
\[
f \mapsto \int_{-1}^1 f(x)K(x,y) dx,
\]
where $K(x,y) >0$ on $[-1,1]^2$, and $dx$ is the Lebesgue measure.
The function $K(x,y)$ is called a kernel function, and we will only consider kernels
of the type
\[
K(x,y) = \phi_0(x)\phi_0(y) + 2 \sum_{k=1}^r g_k \phi_k(x)\phi_k(y),
\]
where $r > 0$ is fixed and the $g_k$s are suitable  scalars such that $K(x,y) >0$ on $[-1,1]^2$.
}
As  the first coefficient  is $g^r_0=1$,   the kernel is normalized:
$\int_{-1}^1 K_r(x,y)dy = T_0(x)/\pi\sqrt{1-x^2}$, and we have: $\int_{-1}^1 f_\KPM^{(r)}(x)dx =\int_{-1}^1 f(x) dx.$
The positivity of the kernel $K_r$ implies that   the integral operator $f\mapsto f_\KPM^{(r)}$ is a positive linear  operator, i.e.,\ a linear operator that maps
the set of nonnegative integrable functions on $[-1,1]$ into itself. Thus the general (Korovkin) convergence theory of
positive linear operators applies and  one may conclude the uniform convergence result
\[
\lim_{r \rightarrow \infty} \|f - f^{(r)}_{KPM}\|^\epsilon_\infty = 0
\]
for any $\epsilon>0$, where  $\|f - f^{(r)}_{KPM}\|^\epsilon_\infty =\max_{-1+\epsilon \le x\le 1-\epsilon} |f(x)-f_\KPM^{(r)}(x)|$. (One needs to restrict the range to subintervals of $[-1,1]$ because of the denominator in the kernel $K_r$.)

In what follows we select  the following  parameters $g^r_k$ for $k=1,\ldots,r$, which define the so-called {Jackson kernel},
again denoted by  $K_r(x,y)$:
\begin{equation}\label{eqgJn}
\begin{split}
g^{r}_k &= \tfrac{1}{r+2}\left((r+2-k) \cos (k\theta_r) + \tfrac{\sin(k\theta_r)}{\sin \theta_r}\cos \theta_r\right) \\
&= \tfrac{1}{r+2} \left((r+2-k) T_k(\cos \theta_r) + U_{k-1}(\cos \theta_r) \, \cos \theta_r\right),
\end{split}
\end{equation}
where we set $$\theta_r:=\frac{\pi}{r+2}.$$
This choice of the parameters $g^r_k$ is the one minimizing the quantity
$\int_{[-1,1]^2} K_r(x,y) (x-y)^2 dx dy,$ which ensures that the corresponding Jackson kernel is maximally peaked at $x=y$ (see \cite[\S II.C.3]{KPM}).

One may  show  that the Jackson kernel $K_r(x,y)$  is indeed positive on $[-1,1]^2$; see  \cite[\S II.C.2]{KPM}.
Moreover $g^r_0=1$ and, for $k=1$, we have $g^{r}_1= \cos(\theta_r)=\cos(\pi/(r+2)) \rightarrow 1$ if $r \rightarrow \infty$ as required.
This is in fact true for all $k$, as will follow from Lemma \ref{lemdJN} below.
Note that one has
$|g^{r}_k| \le 1$ for all $k$, since $|T_k(\cos \theta_r)| \le 1$ and $|U_{k-1}(\cos \theta_r)| \le k$.
For later use, we  now give an estimate on  the Jackson coefficients $g^{r}_k$, showing that $1-g^r_k$ is in the order $O(1/r^2)$. 

\begin{lemma} \label{lemdJN}
Let $d\ge 1$ and $r\ge d$ be given integers, and set $\theta_r = \frac{\pi}{r+2}$.
There exists a constant $C_d$ (depending only on $d$) such that the following inequalities hold:
\[
|1-g^r_k| \le C_d(1-\cos \theta_r) \le \frac{C_d\pi^2}{2(r+2)^2} \ \ \text{ for all } 0\le k\le d.
\]
For the constant $C_d$ we may take $C_d= d^2(1+2c_d)$, where
\begin{equation}\label{eqcd}
c_d=  2^{d-1-2\psi(d)}\frac{d(d-\psi(d)-1)!}{\psi(d)!(d-2\psi(d))!} \ \ \text{ and } \ \ \psi(d)=
\begin{cases}
0 & \text{for } d\leq 4,\\
\left\lceil\frac{1}{8}\left(4d-5 - \sqrt{8d^2-7}\right)\right\rceil & \text{for } d\geq 4.
\end{cases}
\end{equation}
\end{lemma}
\begin{proof}
Define the polynomial
$$P_k(x)=1- \frac{r+2-k}{r+2} T_k(x)- \frac{1}{r+2}xU_{k-1}(x)$$
with degree $k$.
Then,  in view of  relation (\ref{eqgJn}), we have:  $1-g^r_k=P_k(\cos \theta_r)$.
Recall from  relation (\ref{eqval1}) that $T_k(1)=1$ and $U_{k-1}(1)=k$ for any $k\in \oN$. This implies that
$P_k(1)=0$ and thus we can factor $P_k(x)$ as
$P_k(x)=(1-x)Q_k(x)$ for some polynomial $Q_k(x)$ with degree $k-1$.
If we write $P_k(x)=\sum_{i=0}^k p_ix^i$, then it follows that
$Q_k(x)=\sum_{i=0}^{k-1} q_ix^i$, where the scalars $q_i$ are given by
\begin{equation}\label{eqq}
q_i= \sum_{j=0}^i p_j \ \text{ for } i=0,1,\ldots,k-1.
\end{equation}
It now suffices to observe that for any $0\le i\le k$ and $k\le d$, the  $p_i$'s are bounded by a constant depending only on $d$, which will imply that the same holds for the scalars $q_i$.
For this, set
$T_k(x)=\sum_{i=0}^k t_i^{(k)}x^i$ and $U_{k-1}(x)=\sum_{i=0}^{k-1}u_i^{(k)}x^i$. Then the coefficients $p_i$ of $P_k(x)$ can be expressed as
$$p_0= 1-\frac{r+2-k}{r+2} t_0^{(k)},\
p_i= \frac{r+2-k}{r+2} t_i^{(k)} -\frac{u_{i-1}^{(k)}}{r+2} \ (1\le i\le k).$$
For all  $0\le k\le d$ the coefficients  of the Chebyshev polynomials $T_k,U_{k-1}$ can be bounded by an absolute constant depending only on $d$. Namely,
by Lemma \ref{lemChebyshev},
$|t_i^{(k)}|, |u_{i}^{(k)}|\le c_d $ for all $0\le i\le k$ and $k\le d$, where $c_d$ is as defined in (\ref{eqcd}).
As $k\le d\le r$, we have $r+2-k\le r+2$ and thus $|p_i|\le 1+2c_d$ for all $0\le i\le k\le d$.
Moreover, using (\ref{eqq}), $|q_i|\le d(2c_d+1)$ for all $0\le i\le k-1$.
Putting things together we can now derive
$1-g^r_k=(1-\cos \theta_r) Q_k(\cos \theta_r)$, where $Q_k(\cos \theta_r)=\sum_{i=0}^{k-1}q_i (\cos \theta_r)^i$,
so that  $|Q_k(\cos \theta_r)|\le \sum_{i=0}^{k-1}|q_i| \le d^2(1+2c_d)$.
This implies
$|1-g^r_k| \le (1-\cos \theta_r) C_d$, after setting $C_d =d^2(1+2c_d).$
Finally, combining this with  the fact that $1-\cos x \le \frac{x^2}{2}$ for all $x\in [0,\pi]$, we obtain the desired  inequality from the lemma statement.
\qed\end{proof}

\subsection{Jackson kernel approximation of the Dirac delta function}\label{secJacksonapprox}

If one approximates the Dirac delta function $\delta_{x^*}$ at a given point $x^*\in [-1,1]$ by taking its convolution with the Jackson kernel $K_r(x,y)$,
 then the result is the function
\[
\dJN(x-x^*) = \frac{1}{\pi\sqrt{1-x^2}} \left(1 + 2 \sum_{k=1}^{r} g^r_k T_k(x)T_k(x^*)\right);
\]
see \cite[eq. (72)]{KPM}. As mentioned in \cite[eq. (75)--(76)]{KPM}, the function $\dJN$ is in fact a good approximation
to the Gaussian density:
\begin{equation}
\label{Gaussian}  \dJN(x-x^*) \approx  \frac{1}{\sqrt {2\pi\sigma^2}}\mbox{exp}\left(-\frac{(x-x^*)^2}{2\sigma^2}\right)
\mbox{ with }
\sigma^2 \simeq  \left(\frac{\pi}{r+1}\right)^2\left[1-{x^*}^2+ \frac{3{x^*}^2-2}{r+1} \right].
\end{equation}
(Recall that the Dirac delta measure may be defined as a limit of the Gaussian measure when $\sigma \downarrow 0$.)
This approximation is illustrated in Figure \ref{fig:Gaussian approx} for several values of $r$.
\begin{figure}[h!]
\begin{center}
\includegraphics[width=0.5\textwidth]{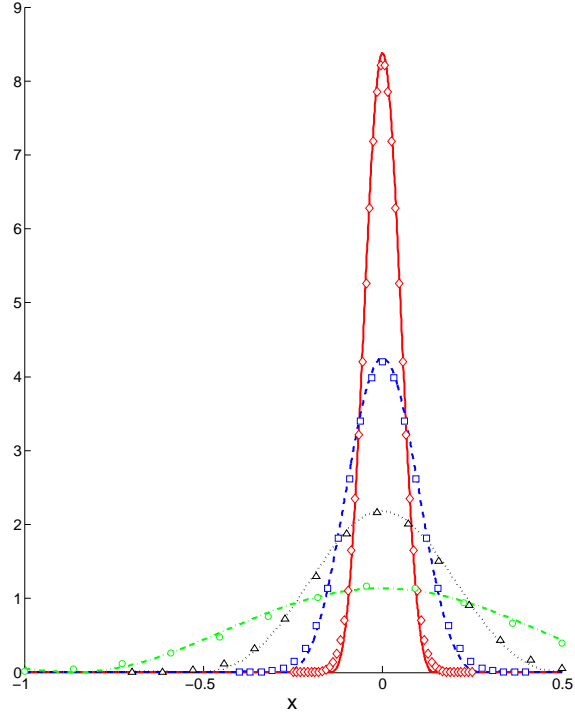}
  \caption{The Jackson kernel approximation $\dJN$ to the Dirac delta function at $x^* = 0$ for $r = 8,16,32,64$. The corresponding scatterplots show the
  values of the Gaussian density function in (\ref{Gaussian}) with $x^* = 0$. \label{fig:Gaussian approx}}
\end{center}
\end{figure}

By construction, the function $\dJN(x-x^*)$ is nonnegative over $[-1,1]$ and  we have the normalization $\int_{-1}^1 \dJN(x-x^*)dx=\int_{-1}^1 \delta_{x^*}(x)dx=1$ (cf. Section \ref{secJacksonkernel}). Hence,
it is a probability density function on $[-1,1]$ for the Lebesgue measure. It is convenient to consider the following univariate polynomial
  \begin{equation}\label{eqhdelta}
 h_r(x)= 1 + 2 \sum_{k=1}^r g^r_k T_k(x)T_k(x^*),
 \end{equation}
 so that $ \dJr(x-x^*) = \frac{1}{\pi\sqrt{1-x^2}} h_r(x).$
  The following facts follow directly, which we will use below for the convergence analysis of the new bounds $f^{(r)}$.

\begin{lemma}\label{lemdJr}
 For any  $r\in \oN$ the polynomial $h_r$ from (\ref{eqhdelta}) is nonnegative over $[-1,1]$ and $\int_{-1}^1 h_r(x) \frac{dx}{\pi\sqrt{1-x^2}}=1$. In other words,
$h_r$   is a probability density function
for the measure $ \left(\pi \sqrt{1-x^2} \right)^{-1}dx$ on $[-1,1]$. 
\end{lemma}

%
%
%

\section{Convergence analysis}\label{secconvergence}

In this section we analyze the convergence rate of the new bounds $f^{(r)}$ and we show the result from Theorem \ref{theodJNn}.
We will first consider the univariate case in Section \ref{secunivariate} (see Theorem \ref{theounivariate}) and then the general multivariate case in Section \ref{secmultivariate} (see Theorem \ref{theomultivariate}).
As we will see, the polynomial $h_r$ arising from the Jackson kernel approximation of the Dirac delta function, introduced above in relation (\ref{eqhdelta}),
 will play a key role in the convergence analysis.


\subsection{ The univariate case}\label{secunivariate}

We consider a univariate polynomial $f$ and let $x^*$ be a global minimizer of $f$ in $[-1,1]$.
As observed in Lemma \ref{lemdJr} the polynomial $h_r$ from (\ref{eqhdelta}) is a density function for the measure $\frac{dx}{ \pi\sqrt{1-x^2}}$.
The key observation now is that the polynomial $h_r$ admits a Schm\"udgen-type representation, of the form
$\sigma(x)+\sigma_1(x) (1-x^2)$ with $\sigma_0,\sigma_1$ sums-of-squares polynomials,  since it is non-negative over $[-1,1]$.  This fact  will allow us to use the polynomial $h_r$ to get feasible solutions for the program defining the bound $f^{(r)}$. It
follows from the following
classical result
(see, e.g., \cite{PR}), 
 that characterizes  univariate polynomials that are nonnegative on $[-1,1]$.
(Note that this is a strengthening of  Schm\"udgen's theorem (Theorem \ref{thm:Schmuedgen}) in the univariate case.)
\begin{theorem}[Fekete, Markov-Luk\`acz]
\label{theoFekete}
 Let $p(x)$ be a univariate polynomial of degree
$m$. Then $p(x)$ is nonnegative on the interval $[-1,1]$ if and
only if it has the following representation:
\begin{equation*}
p(x)=\sigma_0(x)+(1-x^2)\sigma_1(x)
\end{equation*}
for some sum-of-squares polynomials $\sigma_0$ of degree $2\lceil m/2\rceil$ and $\sigma_1$ of degree $2\lceil m/2\rceil -2$.
\end{theorem}


\noindent
We start with the following technical lemma.

%

\begin{lemma}\label{lemtech1}
Let $f$ be a polynomial of degree $d$ written in the Chebyshev basis as $f=\sum_{k=0}^df_kT_k$,
let $x^*$ be a global minimizer of $f$ in $[-1,1]$, and let $h_r$ be the polynomial from (\ref{eqhdelta}).
For any integer $r\ge d$ we have
$$\int_{-1}^1 f(x)h_r(x)\frac{dx}{\pi\sqrt{1-x^2}}  - f(x^*) \le\frac{C_f}{(r+2)^2},$$
where $C_f= (\sum_{k=1}^d |f_k|) \frac{C_d\pi^2}{2}$ and $C_d$ is the constant from Lemma \ref{lemdJN}.
\end{lemma}

\begin{proof}
As $f=\sum_{k=0}^d f_kT_k$ and $h_r=1+2 \sum_{k=1}^r g_k^r T_k(x^*) T_k$,  we use the orthogonality relationships (\ref{eqortho}) to obtain
\begin{equation}\label{eqfhr1}
\int_{-1}^1 f(x)h_r(x) \frac{dx}{\pi \sqrt{1-x^2} }  = \sum_{k=0}^d f_k T_k(x^*) g^r_k.
\end{equation}
Combining with $f(x^*)=\sum_{k=0}^d f_k T_k(x^*)$ gives
\begin{equation}\label{eqfhr2}
\int_{-1}^1 f(x)h_r(x) \frac{dx}{ \pi \sqrt{1-x^2} } -f(x^*)=
\sum_{k=1}^d f_kT_k(x^*) (g_k^r-1).
\end{equation}
Now we use the upper bound on $g_k^r-1$ from Lemma \ref{lemdJN} and the bound $|T_k(x^*)|\le 1$ to conclude the proof.
\qed\end{proof}

We can now conclude the convergence analysis of the bounds $f^{(r)}$ in the univariate case.

\begin{theorem}\label{theounivariate}
Let $f=\sum_{k=0}^df_kT_k$ be a polynomial of degree $d$.
For any integer $r\ge d$ we have
$$f^{(r)}   - f_{\min}  \le\frac{C_f}{ (r+1)^2},$$
where $C_f=(\sum_{k=1}^d |f_k|) \frac{C_d\pi^2}{ 2}$ and $C_d$ is the constant from Lemma \ref{lemdJN}.
\end{theorem}

\begin{proof}
Using  the degree bounds in Theorem \ref{theoFekete} for the sum-of-squares polynomials entering the decomposition of the polynomial $h_r$, we can conclude that for $r$ even, $h_r$ is feasible for the program  defining the parameter $f^{(r)}$ and for $r$ odd,   $h_r$ is feasible for the program defining the parameter $f^{(r+1)}$.
Setting $C_f= (\sum_{k=1}^d |f_k|) \frac{C_d\pi^2}{ 2}$ and using Lemma \ref{lemtech1}, this implies:
$f^{(r)}-\fmin \le \frac{C_f}{ (r+2)^2}$ for $r$ even, and
$f^{(r)}-\fmin \le \frac{C_f}{ (r+1)^2}$ for odd $r$.
The result of the theorem now follows.  \qed \end{proof}


\subsection{The multivariate case}\label{secmultivariate}

%

We consider now a multivariate polynomial $f$ and we let $x^* = (x^*_1,\dotsc,x^*_n) \in [-1,1]^n$ denote a global minimizer of $f$ on $[-1,1]^n$, i.e., $f(x^*) = f_{\min}$.
\newcommand{\ur}{{\underline r}}

In order to obtain a feasible solution to the program defining the parameter $f^{(r)}$ we will consider products of the univariate polynomials $h_r$ from (\ref{eqhdelta}).
Namely, given integers $r_1,\ldots,r_n \in \oN$ we define the $n$-tuple $\ur=(r_1,\ldots,r_n)$ and the $n$-variate polynomial
\begin{equation}\label{eqHdelta}
H_{\ur}(x_1,\ldots,x_n)=\prod_{i=1}^n h_{r_i}(x_i).
\end{equation}
We group in the next lemma some properties of the polynomial $H_{\ur}$.
\begin{lemma}\label{lemHr}
The polynomial $H_\ur$ satisfies the following properties:
\begin{itemize}
\item[(i)] $H_{\ur}$ is non-negative on $[-1,1]^n$.
\item[(ii)] $\int_{[-1,1]^n}H_{\ur}(x) d\mu(x)=1$, where $d\mu$ is the measure from (\ref{eqmu}).
\item[(iii)]
$H_{\ur}$ has a Schm\"udgen-type representation of the form $H_{\ur}(x)=\sum_{I\subseteq [n]}\sigma_I(x)\prod_{i\in I}(1-x_i^2)$, where each $\sigma_I$ is a sum-of-squares polynomial of degree at most $2\sum_{i=1}^n \lceil r_i/2\rceil -2|I|$.
\end{itemize}
\end{lemma}

\begin{proof}
(i) and (ii) follow directly from the corresponding properties of the univariate polynomials $h_{r_i}$, and (iii) follows using Theorem \ref{theoFekete} applied to the polynomials $h_{r_i}$.
\qed\end{proof}

The next lemma is the analog of Lemma \ref{lemtech1} for the multivariate case.

\begin{lemma}\label{lemtech3}
Let $f$ be a multivariate polynomial of degree $d$, written in the basis of multivariate Chebyshev polynomials  as
$f=\sum_{\alpha\in\oN^n:|\alpha|\le d} f_\alpha T_\alpha$, and let $x^*$ be a global minimizer of $f$ in $[-1,1]^n$.
Consider $\ur=(r_1,\ldots,r_n)$, where each $r_i$ is an integer satisfying $r_i\ge d$, and the polynomial $H_{\ur}$ from (\ref{eqHdelta}). We have
$$\int_{[-1,1]^n} f(x)H_{\ur}(x)d\mu(x) - f(x^*)\le C_f \sum_{i=1}^n \frac{1}{ (r_i+2)^2},$$
where $C_f=(\sum_{\alpha: |\alpha|\le d} |f_\alpha|)\frac{C_d\pi^2}{ 2}$ and $C_d$ is the constant from Lemma \ref{lemdJN}.
\end{lemma}

\begin{proof}
As $f=\sum_{\alpha: |\alpha|\le d}f_\alpha T_\alpha$ and
$H_{\ur}=\prod_{i=1}^n h_{r_i}(x_i)=\prod_{i=1}^n (1+2\sum_{k_i=1}^{r_i} g_{k_i}^{r_i} T_{k_i}(x_i)T_{k_i}(x_i^*))$, we can use the orthogonality relationships
(\ref{eqorthomulti}) among the multivariate
Chebyshev polynomials to derive
$$\int_{[-1,1]^n}f(x)H_{\ur}(x)d\mu(x) = \sum_{\alpha: |\alpha|\le d} f_\alpha T_\alpha(x^*) \prod_{i=1}^n g_{\alpha_i}^{r_i}.$$
Combining this with $f(x^*)=\sum_{\alpha: |\alpha|\le d} f_\alpha T_\alpha(x^*)$ gives:
$$\int_{[-1,1]^n}f(x)H_{\ur}(x)d\mu(x) -f(x^*)= \sum_{\alpha: |\alpha|\le d} f_\alpha T_\alpha(x^*)( \prod_{i=1}^n g_{\alpha_i}^{r_i}-1).$$
Using the identity:  $ \prod_{i=1}^n g_{\alpha_i}^{r_i}-1=\sum_{j=1}^{n} (g^{r_j}_{\alpha_j}-1) \prod_{k=j+1}^n g^{r_k}_{\alpha_k}$ and the fact that $|g^{r_k}_{\alpha_k}|\le 1$,
we get $| \prod_{i=1}^n g_{\alpha_i}^{r_i}-1|\le \sum_{j=1}^n |g^{r_j}_{\alpha_j}-1|$.
Now use  $|T_{\alpha}(x^*)|\le 1$ and the bound from Lemma \ref{lemdJN} for each $|1-g_{\alpha_j}^{r_j}|$ to conclude the proof.
\qed\end{proof}

We can now show our main result, which implies Theorem \ref{theodJNn}.

\begin{theorem}\label{theomultivariate}
Let $f=\sum_{\alpha: |\alpha|\le d} f_\alpha T_\alpha$  be an $n$-variate polynomial of degree $d$. For any integer $r\ge n(d+2)$, we have
$$f^{(r)} -\fmin \le \frac{C_f n^3}{ (r+1)^2},$$
where $C_f=(\sum_{\alpha: |\alpha|\le d} |f_\alpha|) \frac{C_d\pi^2}{ 2}$ and $C_d$ is the constant from Lemma \ref{lemdJN}.
\end{theorem}

\begin{proof}
Write $r-n=sn+n_0$, where $s,n_0\in\oN$ and $0\le n_0<n$, and define the $n$-tuple $\ur=(r_1,\ldots,r_n)$, setting $r_i=s+1$ for $1\le i\le n_0$ and $r_i=s$ for $n_0+1\le i\le n$, so that $r-n=r_1+\ldots +r_n$. Note that the condition $r\ge n(d+2)$ implies $s\ge d$ and thus $r_i\ge d$ for all $i$.
Moreover, we have: $2\sum_{i=1}^n \lceil r_i/2\rceil = 2n_0\lceil (s+1)/2\rceil+2(n-n_0) \lceil s/2\rceil$, which is equal to $r-n+n_0$ for even $s$ and to $r-n_0$ for odd $s$ and thus always at most $r$. Hence the polynomial $H_{\ur}$ from (\ref{eqHdelta}) has degree at most $r$. By Lemma \ref{lemHr}(ii), (iii), it follows that the polynomial $H_{\ur}$ is feasible for the program defining the parameter $f^{(r)}$.
By Lemma \ref{lemtech3} this implies that
$$f^{(r)} -\fmin \le \int_{[-1,1]^n} f(x)H_{\ur}(x)d\mu(x) - f(x^*)\le C_f \sum_{i=1}^n \frac{1}{ (r_i+2)^2}.$$
Finally, $\sum_{i=1}^n\frac{1}{ (r_i+2)^2} =\frac{n_0}{ (s+3)^2}+\frac{n-n_0}{ (s+2)^2}\le \frac{n}{ (s+2)^2} =\frac{n^3}{ (r+n-n_0)^2}\le \frac{n^3}{ (r+1)^2}$, since $n_0\le n-1$.
\qed\end{proof}

\section{Computing the parameter $f^{(r)}$ as a generalized eigenvalue problem}\label{seceigenvalue}

As the parameter $f^{(r)}$ is defined in terms of sum-of-squares polynomials (cf. Definition \ref{defnewbound}), it can be computed by means of a semidefinite program.
As we now observe, as the program (\ref{fr}) has only one affine constraint,  $f^{(r)}$ can in fact be computed in a cheaper way as a generalized eigenvalue problem.

 Using the inner product from (\ref{eqinner}), the parameter $f^{(r)}$   can be rewritten as
\begin{equation}\label{eqbound1}
\begin{array}{ll}
\displaystyle
f^{(r)}=\min_{h\in \oR[x]} \langle f,h\rangle \ \text{ such that } & \langle h,T_0\rangle =1, \  h(x)=\sum_{I\subseteq [n]} \sigma_I (x) \prod_{i\in I} (1-x_i^2),\\
& \sigma_I\in \Sigma[x],\  \deg(\sigma_I)\le  r-2|I| \ \; \forall  I\subseteq [n].
\end{array}
\end{equation}

For convenience we use below the following notation. For a set $I\subseteq [n]$ and an integer $r\in \oN$ we let $\Lambda^I_r$ denote the set of sequences $\beta\in\oN^n$ with $|\beta|\le \lfloor \frac{r-2|I|}{ 2}\rfloor$.
As is well known one can express the condition that $\sigma_I$ is a sum-of-squares polynomial, i.e., of the form $\sum_k p_k(x)^2$ for some $p_k\in \oR[x]$,  as a semidefinite program. More precisely, using the Chebyshev basis  to express the polynomials $p_k$,
we obtain that $\sigma_I$ is a sum-of-squares polynomial if and only if there exists  a matrix variable $M^I$ indexed by $\Lambda^I_r$,
which is positive semidefinite and satisfies
\begin{equation}\label{eqsigmaI}
\sigma_I= \sum_{\beta,\gamma\in \Lambda^I_r} M^I_{\beta,\gamma} T_\beta T_\gamma.
\end{equation}
For each $I\subseteq [n]$, we introduce the following matrices $A^I$ and $B^I$, which are also  indexed by the set $\Lambda^I_r$ and, for $\beta,\gamma\in \Lambda^I_r$,  with entries
\begin{equation}\label{eqAB}
\begin{array}{l}
\displaystyle A^I_{\beta,\gamma} = \langle f, T_\beta T_\gamma \prod_{i\in I} (1-x_i^2)\rangle,\\
\displaystyle B^I_{\beta,\gamma} = \langle T_0,T_\beta T_\gamma \prod_{i\in I} (1-x_i^2)\rangle.
\end{array}
\end{equation}
We will indicate in the appendix how to compute the matrices $A^I$ and $B^I$.

We can now reformulate the parameter $f^{(r)}$ as follows.

\begin{lemma}
Let $A^I$ and $B^I$ be the matrices defined as in (\ref{eqAB}) for each $I\subseteq [n]$. Then
the parameter $f^{(r)}$ can be reformulated  using the following semidefinite program in the matrix variables $M^I$ ($I\subseteq [n]$):
\begin{equation}\label{eqbound2}
f^{(r)}=\min_{M^I: I\subseteq [n]} \ \sum_{I\subseteq [n]} \tr(A^IM^I) \ \text{ such that }\ M^I\succeq 0 \ \forall I\subseteq [n],\ \sum_{I\subseteq [n]} \tr(B^IM^I) =1.
\end{equation}
\end{lemma}

\begin{proof}
Using  relation (\ref{eqsigmaI}) we can express the polynomial variable $h$ in (\ref{eqbound1}) in terms of the matrix variables $M^I$ and obtain
$$h= \sum_{I\subseteq [n]}  \sum_{\beta,\gamma\in \Lambda^I_r}M^I_{\beta,\gamma} T_\beta T_\gamma \prod_{i\in I}(1-x_i)^2.$$
First this permits us to reformulate the objective function $\langle f,h\rangle$ in terms of the matrix variables $M^I$ in the following way:
$$\begin{array}{ll}
\langle f,h\rangle
& =\sum_I\sum_{\beta,\gamma} M^I_{\beta,\gamma} \langle f,T_\beta T_\gamma\prod_{i\in I}(1-x_i^2)\rangle\\
& = \sum_I\sum_{\beta,\gamma} M^I_{\beta,\gamma} A^I_{\beta,\gamma}\\
& =\sum_I\tr(A^IM^I).
\end{array}
$$
Second we can reformulate the constraint $\langle T_0,h\rangle =1$ using
$$\begin{array}{ll}
\langle T_0,h\rangle
& = \sum_I \sum_{\beta,\gamma} M^I_{\beta,\gamma} \langle T_0, T_\beta T_\gamma \prod_{i\in I}(1-x_i^2)\rangle \\
&= \sum_I\sum_{\beta,\gamma} M^I_{\beta,\gamma} B^I_{\beta,\gamma} \\
&= \sum_I \tr(B^IM^I).
\end{array}
$$
From this follows that the program (\ref{eqbound1}) is indeed equivalent to the program (\ref{eqbound2}).
\qed\end{proof}

The program (\ref{eqbound2})  is a semidefinite program  with only one constraint. Hence, as we show next,  it is  equivalent to a generalized eigenvalue problem.

\begin{theorem}
\label{th:gen eig formulation}
For $I \subseteq [n]$ let $A^I$ and $B^I$ be the matrices from (\ref{eqAB}) and define the parameter
\[
\lambda^{(I)} = \max \left\{\lambda \; |\; A^I - \lambda B^{I} \succeq 0\right\}= \min \left\{\lambda \; |\; A^Ix = \lambda B^{I}x \mbox{ for some non-zero vector $x$}\right\}.
\]
One then has
$f^{(r)} = \min_{I \subseteq [n]} \lambda^{(I)}$.
\end{theorem}

\proof
The dual semidefinite program of the program  (\ref{eqbound2}) is given by
\begin{equation}\label{eqdual}
\sup \left\{  \lambda \; |\; A^{I} - \lambda B^{I} \succeq 0 \;\;\; \forall I \subseteq [n] \right\}.
\end{equation}
We first show that the primal problem (\ref{eqbound2}) is strictly feasible. To see this it suffices to show that   $\tr(B^I) > 0$, since then  one may
 set $M_I$ equal to a suitable multiple of the identity matrix and thus one gets a strictly feasible solution to (\ref{eqbound2}).
Indeed, the matrix $B^I$ is positive semidefinite since, for any  scalars $g_\beta$,
$$\sum_{\beta,\gamma}g_\beta g_\gamma B^I_{\beta \gamma} =
\int_{[-1,1]^n} \Bigl(\sum_\beta g_\beta T_\beta\Bigr)^2 \prod_{i\in I}(1-x_i^2) d\mu(x)\ge 0.$$
Thus $\tr(B^I)\ge 0$ and, moreover, $\tr(B^I)>0$ since
$B^I$ is nonzero.

Moreover, the dual problem (\ref{eqdual}) is also feasible, since $\lambda=f_{\min}$ is a feasible solution.
This follows from the fact that the polynomial $f-f_{\min}$ is nonnegative over $[-1,1]^n$, which implies that the matrix $A^I-f_{\min} B^I$ is positive semidefinite. Indeed, using the same argument as above for showing that $B^I\succeq 0$, we have
$$\sum_{\beta,\gamma}g_\beta g_\gamma (A^I-f_{\min} B^I)_{\beta,\gamma}=\int_{[-1,1]^n} (f(x)-f_{\min})\Bigl(\sum_\beta g_\beta T_\beta\Bigr)^2 \prod_{i\in I}(1-x_i^2) d\mu(x) \ge 0.$$
Since the primal problem is strictly feasible and the dual problem is feasible, there is no duality gap and the dual problem attains its supremum. The result follows.
\qed

\section{Numerical examples}\label{secnumerical}

We examine the  polynomial test functions which were also used in \cite{ConvAna} and \cite{ElemCalc}, and are described in the appendix to this paper.

The numerical examples given here only serve to illustrate the observed convergence behavior of the sequence $f^{(r)}$ as compared to the theoretical convergence rate. In particular, the computational demands for computing $f^{(r)}$ for large $r$ are such that it cannot compete in practice with the known iterative methods referenced in the introduction.

For the polynomial test functions we list in (Table \ref{tab:fN results}) the values of $f^{(r)}$ for even $r$ up to $r=48$, obtained by solving the generalized eigenvalue problem in Theorem \ref{th:gen eig formulation}
using the {\tt eig} function of Matlab. Recall that for step
$r$ of the hierarchy the polynomial density function $h$ is of Schm\"udgen-type and  has degree $r$.

For the examples listed the computational time is negligible, and therefore not listed;
recall that the computation of $f^{(r)}$ for even $n$ requires the solution of $2^n$ generalized eigenvalue problems indexed by subsets $I \subset [n]$,
where the order of the matrices equals $\binom{n+{\lfloor r/2 -|I| \rfloor} }{ n}$;
cf.\ Theorem \ref{th:gen eig formulation}.

\begin{table}[h!]
\caption{The upper bounds $f^{(r)}$ for the test functions.\label{tab:fN results}}
\begin{center}
{\small
\begin{tabular}{|c|c|c|c|c|c|c|c|c|}
	\hline	
	\multirow{2}{*}{$r$} & \multirow{2}{*}{Booth} & \multirow{2}{*}{Matyas} & \multirow{2}{*}{Motzkin} & \multirow{2}{*}{Three-Hump} & \multicolumn{2}{c|}{Styblinski-Tang} & \multicolumn{2}{c|}{Rosenbrock}\\
	& & & & & $n=2$ & $n=3$ & $n=2$ & $n=3$ \\
	\hline
    6    &    145.3633  &   4.1844  &   1.1002 &   24.6561 &  -27.4061  &              &  157.7604     &                \\
    8    &    118.0554  &   3.9308  &   0.8764 &   15.5022 &  -34.5465  &   -40.1625   &   96.8502     & 318.0367                \\
   10    &     91.6631  &   3.8589  &   0.8306 &    9.9919 &  -40.0362  &   -47.6759   &   68.4239     & 245.9925       \\
   12    &     71.1906  &   3.8076  &   0.8098 &    6.5364 &  -47.4208  &   -55.4061   &   51.7554     & 187.2490       \\
   14    &     57.3843  &   3.0414  &   0.7309 &    4.5538 &  -51.2011  &   -64.0426   &   39.0613     & 142.8774       \\
   16    &     47.6354  &   2.4828  &   0.6949 &    3.3453 &  -56.0904  &   -70.2894   &   30.3855     & 111.0703       \\
   18    &     40.3097  &   2.0637  &   0.5706 &    2.5814 &  -58.8010  &   -76.0311   &   24.0043     &  88.3594        \\
   20    &     34.5306  &   1.7417  &   0.5221 &    2.0755 &  -61.8751  &   -80.5870   &   19.5646     &  71.5983        \\
   22    &     28.9754  &   1.4891  &   0.4825 &    1.7242 &  -63.9161  &   -85.4149   &   16.2071     &  59.0816        \\
   24    &     24.6380  &   1.2874  &   0.4081 &    1.4716 &  -65.5717  &   -88.5665   &   13.6595     &  49.5002        \\
   26    &     21.3151  &   1.1239  &   0.3830 &    1.2830 &  -67.2790  &              &   11.6835     &        \\
   28    &     18.7250  &   0.9896  &   0.3457 &    1.1375 &  -68.2078  &              &   10.1194     &                 \\
   30    &     16.6595  &   0.8779  &   0.3016 &    1.0216 &  -69.5141  &              &    8.8667     &                 \\
   32    &     14.9582  &   0.7840  &   0.2866 &    0.9263 &  -70.3399  &              &    7.8468     &                 \\
   34    &     13.5114  &   0.7044  &   0.2590 &    0.8456 &  -71.0821  &              &    7.0070     &                  \\
   36    &     12.2479  &   0.6363  &   0.2306 &    0.7752 &  -71.8284  &              &    6.3083     &                  \\
   38    &     11.0441  &   0.5776  &   0.2215 &    0.7129 &  -72.2581  &              &    5.7198     &                  \\
   40    &     10.0214  &   0.5266  &   0.2005 &    0.6571 &  -72.8953  &              &    5.2215     &                  \\
   42    &      9.1504  &   0.4821  &   0.1815 &    0.6070 &  -73.3011  &              &    4.7941     &                  \\
   44    &      8.4017  &   0.4430  &   0.1754 &    0.5622 &  -73.6811  &              &    4.4266     &                  \\
   46    &      7.7490  &   0.4084  &   0.1597 &    0.5220 &  -74.0761  &              &    4.1070     &                  \\
   48    &      7.1710  &   0.3778  &   0.1462 &    0.4860 &  -74.3070  &              &    3.8283     &                  \\
 	\hline
\end{tabular}
}\end{center}
\end{table}

We note that the observed rate of convergence seems in line with the $O(1/r^2)$ error bound.

As a second numerical experiment, we compare (see Table \ref{table:result1}) the upper bound $f^{(r)}$ to the upper bound
$\underline{f}_{\mathbf{K}}^{(r)}$ defined in (\ref{fundr}). Recall that the bound $\underline{f}_{\mathbf{K}}^{(r)}$
corresponds to using sum-of-squares density functions of degree at most $r$ and the Lebesgue measure.
As shown in \cite{ConvAna}, the computation of $\underline{f}_{\mathbf{K}}^{(r)}$ may be done by solving a single
generalized eigenvalue problem with matrices of order $\binom{n+{\lfloor r/2 -|I| \rfloor} }{ n}$.
Thus the computation of  $\underline{f}_{\mathbf{K}}^{(r)}$ is significantly cheaper than that of $f^{(r)}$.

\begin{table}[h!]
\captionsetup{width=13cm}
\caption{Comparison of the upper bounds $f^{(r)}$ and $\underline{f}_{\mathbf{K}}^{(r)}$ for Booth, Matyas, Three--Hump Camel, and Motzkin functions.
\label{table:result1}}
\begin{center}
{\small
	\begin{tabular}{| c | >{\centering}m{1.2cm} | c | >{\centering}m{1cm} | c | >{\centering}m{1.3cm} | c | >{\centering}m{1.2cm} | c |}
    \hline
\multirow{2}{*}{$r$} & \multicolumn{2}{c|}{Booth function} & \multicolumn{2}{c|}{Matyas function} & \multicolumn{2}{m{3cm}|}{\centering Three--Hump Camel function}& \multicolumn{2}{c|}{Motzkin polynomial} \\ \cline{2-9}
                   & $\underline{f}_{\mathbf{K}}^{(r)}$  & $f^{(r)}$& $\underline{f}_{\mathbf{K}}^{(r)}$  & $f^{(r)}$  & $\underline{f}_{\mathbf{K}}^{(r)}$  & $f^{(r)}$  & $\underline{f}_{\mathbf{K}}^{(r)}$  & $f^{(r)}$\\ \hline
$6$  & 118.383 &   145.3633 & 4.2817 & 4.1844 & 29.0005 & 24.6561    & 1.0614 & 1.1002\\ \hline
$8$  & 97.6473 &   118.0554 & 3.8942 & 3.9308 & 9.5806 & 15.5022     & 0.8294 & 0.8764\\ \hline
$10$ & 69.8174 &    91.6631 & 3.6894 & 3.8589 & 9.5806 &  9.9919     & 0.8010 & 0.8306\\ \hline
$12$ & 63.5454 &    71.1906 & 2.9956 & 3.8076 & 4.4398 &  6.5364     & 0.8010 & 0.8098\\ \hline
$14$ & 47.0467 &    57.3843 & 2.5469 & 3.0414 & 4.4398 &  4.5538     & 0.7088 & 0.7309\\ \hline
$16$ & 41.6727 &    47.6354 & 2.0430 & 2.4828 & 2.5503 &  3.3453     & 0.5655 & 0.6949\\ \hline
$18$ & 34.2140 &    40.3097 & 1.8335 & 2.0637 & 2.5503 &  2.5814     & 0.5655 & 0.5706\\ \hline
$20$ & 28.7248 &    34.5306 & 1.4784 & 1.7417 & 1.7127 &  2.0755     & 0.5078 & 0.5221\\ \hline
$22$ & 25.6050 &    28.9754 & 1.3764 & 1.4891 & 1.7127 &  1.7242     & 0.4060 & 0.4825\\ \hline
$24$ & 21.1869 &    24.6380 & 1.1178 & 1.2874 & 1.2775  &  1.4716    & 0.4060 & 0.4081\\ \hline
$26$ & 19.5588 &    21.3151 & 1.0686 & 1.1239 & 1.2775  &  1.2830    & 0.3759   & 0.3830\\ \hline
$28$ & 16.5854 &    18.7250 & 0.8742 & 0.9896 & 1.0185  &  1.1375    & 0.3004   & 0.3457\\ \hline
$30$ & 15.2815 &    16.6595 & 0.8524 & 0.8779 & 1.0185  &  1.0216    & 0.3004   & 0.3016\\ \hline
$32$ & 13.4626 &    14.9582 & 0.7020 & 0.7840 & 0.8434  &  0.9263    & 0.2819   & 0.2866\\ \hline
$34$ & 12.2075 &    13.5114 & 0.6952 & 0.7044 & 0.8434  &  0.8456    & 0.2300   & 0.2590\\ \hline
$36$ & 11.0959 &    12.2479 & 0.5760 & 0.6363 & 0.7113  &  0.7752    & 0.2300   & 0.2306\\ \hline
$38$ & 9.9938  &    11.0441 & 0.5760 & 0.5776 & 0.7113  &  0.7129    & 0.2185   & 0.2215\\ \hline
$40$ & 9.2373  &    10.0214 & 0.4815 & 0.5266 & 0.6064  &  0.6571    & 0.1817   & 0.2005\\ \hline
  \end{tabular}
  }\end{center}
  \end{table}
					
 It is interesting to note that, in almost all cases, $f^{(r)} > \underline{f}_{\mathbf{K}}^{(r)}$.
 Thus even though the measure $d\mu(x)$ and the Schm\"udgen-type densities are useful in getting improved error bounds,
 they mostly do not lead to improved upper bounds for these examples.
 This also suggests that it might be possible to improve the error result $\underline{f}_{\mathbf{K}}^{(r)} - f_{\min} = O(1/\sqrt{r})$  in \cite{ConvAna},
 at least for the case ${\mathbf{K}} = [-1,1]^n$. To illustrate this effect we graphically represented the results of Table \ref{table:result1} in Figure \ref{fig:result1}. Note that the bound $\frac{C_fn^3}{(r+1)^2}$ of Theorem \ref{theomultivariate} would lie far above these graphs. To give an idea for the value of the constants $C_f$ we calculated them for the Booth, Matyas, Three-Hump Camel,
 and Motzkin functions: $C_{\text{Booth}} \approx  2.6\cdot 10^5,\ C_{\text{Matyas}} \approx 9.9\cdot 10^3,\ C_{\text{ThreeHump}} \approx 3.5\cdot 10^7$ and $C_{\text{Motzkin}} \approx 1.1\cdot 10^5$.

\begin{figure}[h!]
\begin{center}
\begin{subfigure}{0.4\textwidth}
	\includegraphics[width=\textwidth]{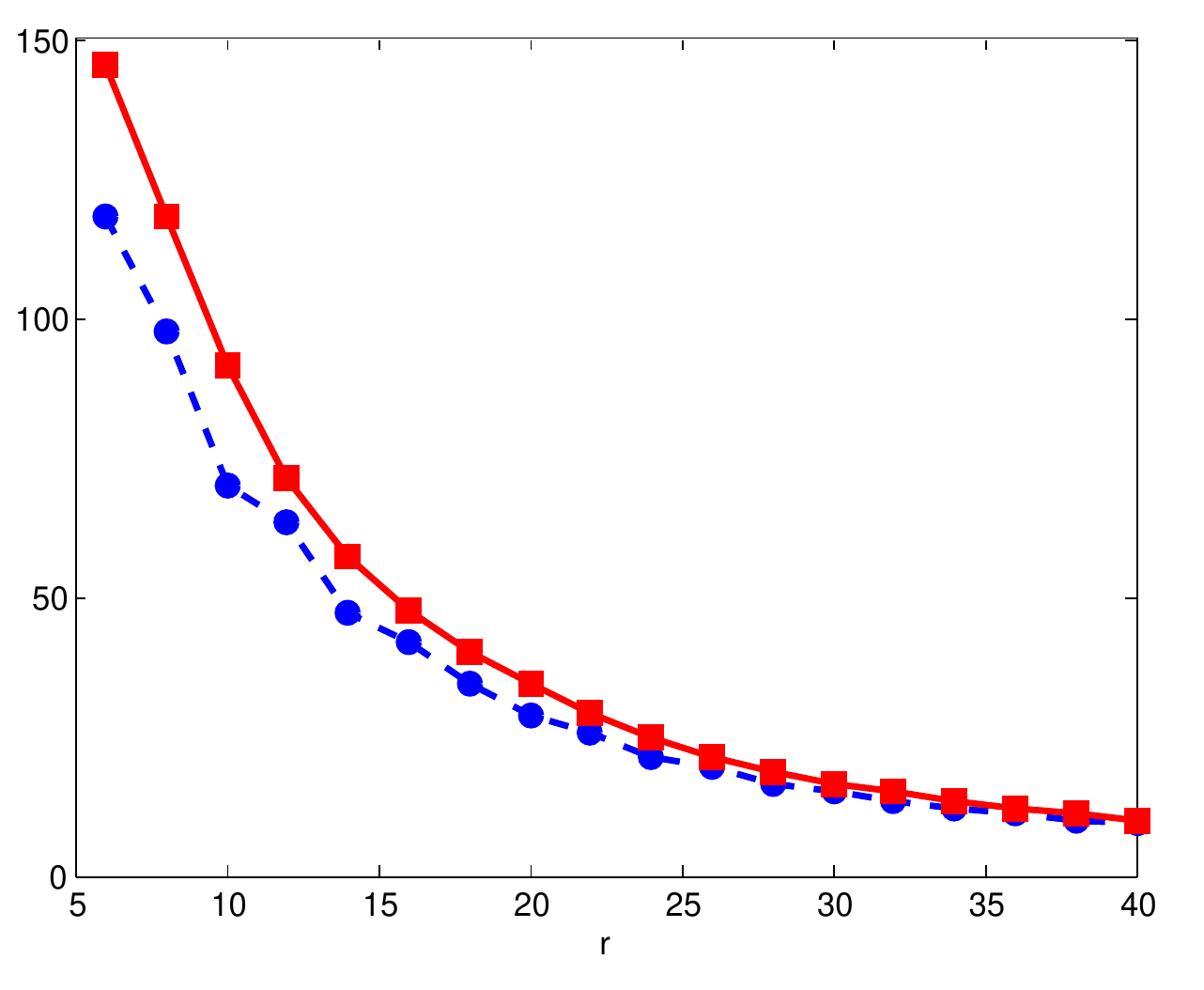}
	\caption{Booth function}
\end{subfigure}
\begin{subfigure}{0.4\textwidth}
	\includegraphics[width=\textwidth]{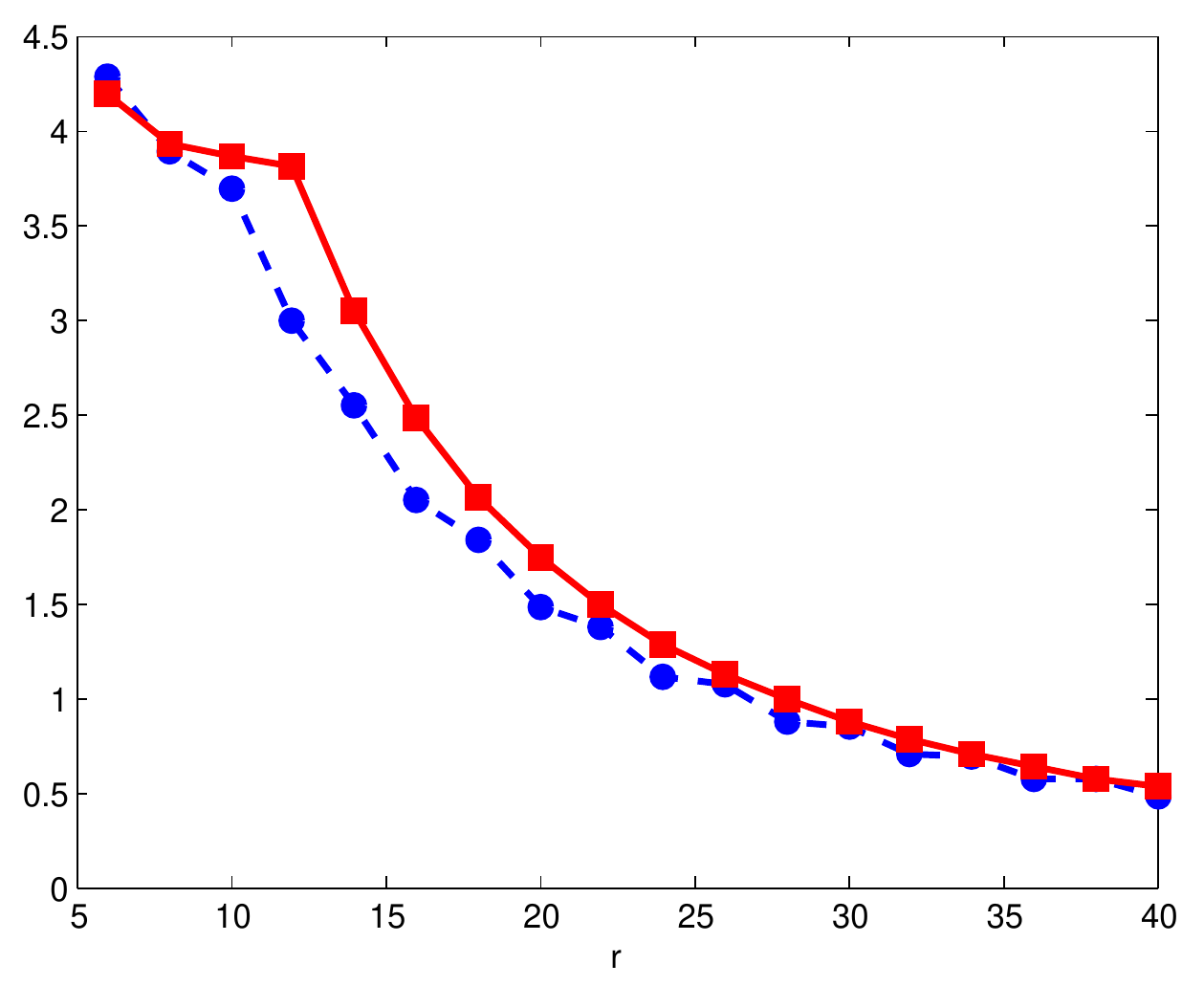}
	\caption{Matyas function}
\end{subfigure}
\begin{subfigure}{0.4\textwidth}
	\includegraphics[width=\textwidth]{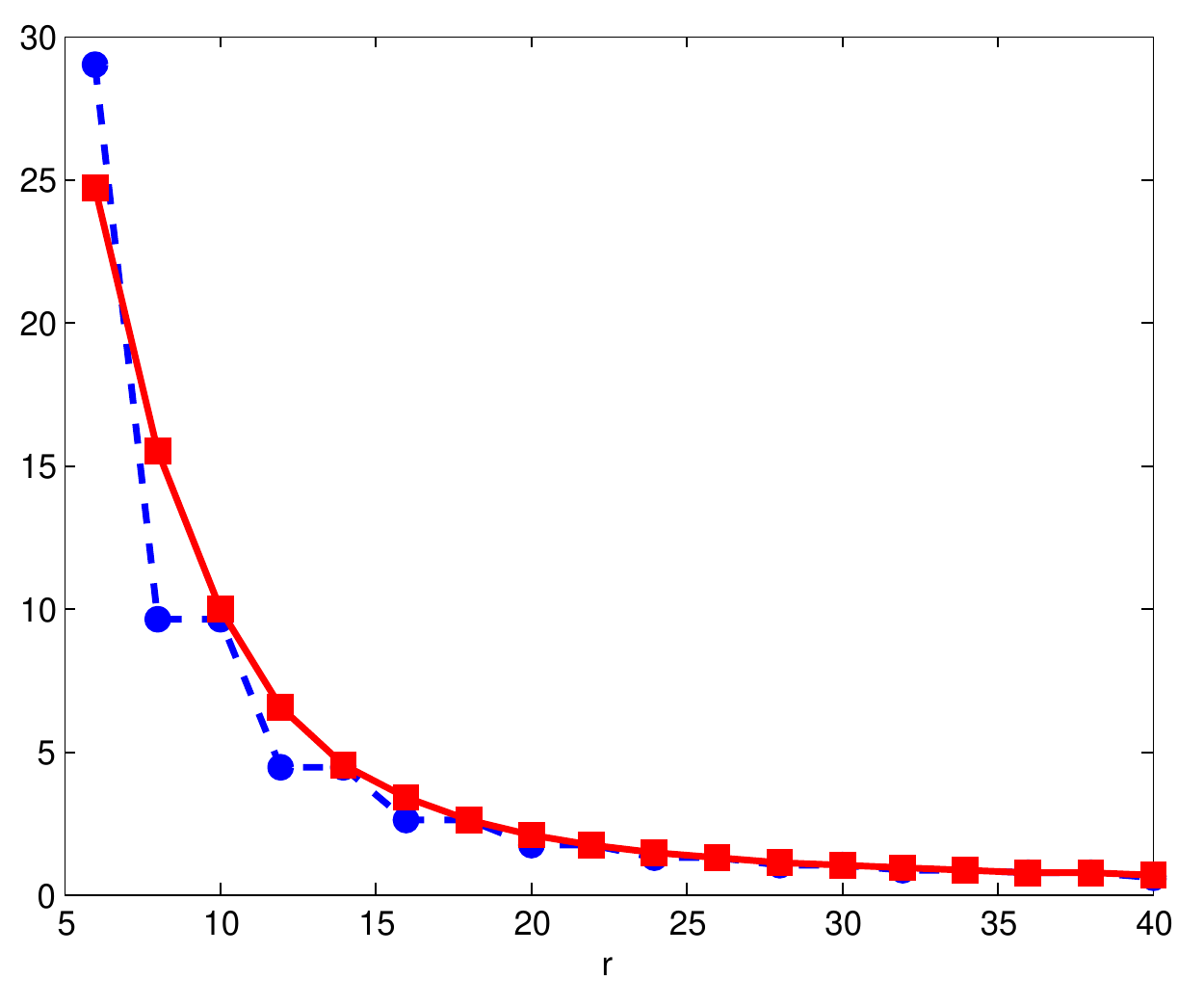}
	\caption{Three-Hump Camel function}
\end{subfigure}
\begin{subfigure}{0.4\textwidth}
	\includegraphics[width=\textwidth]{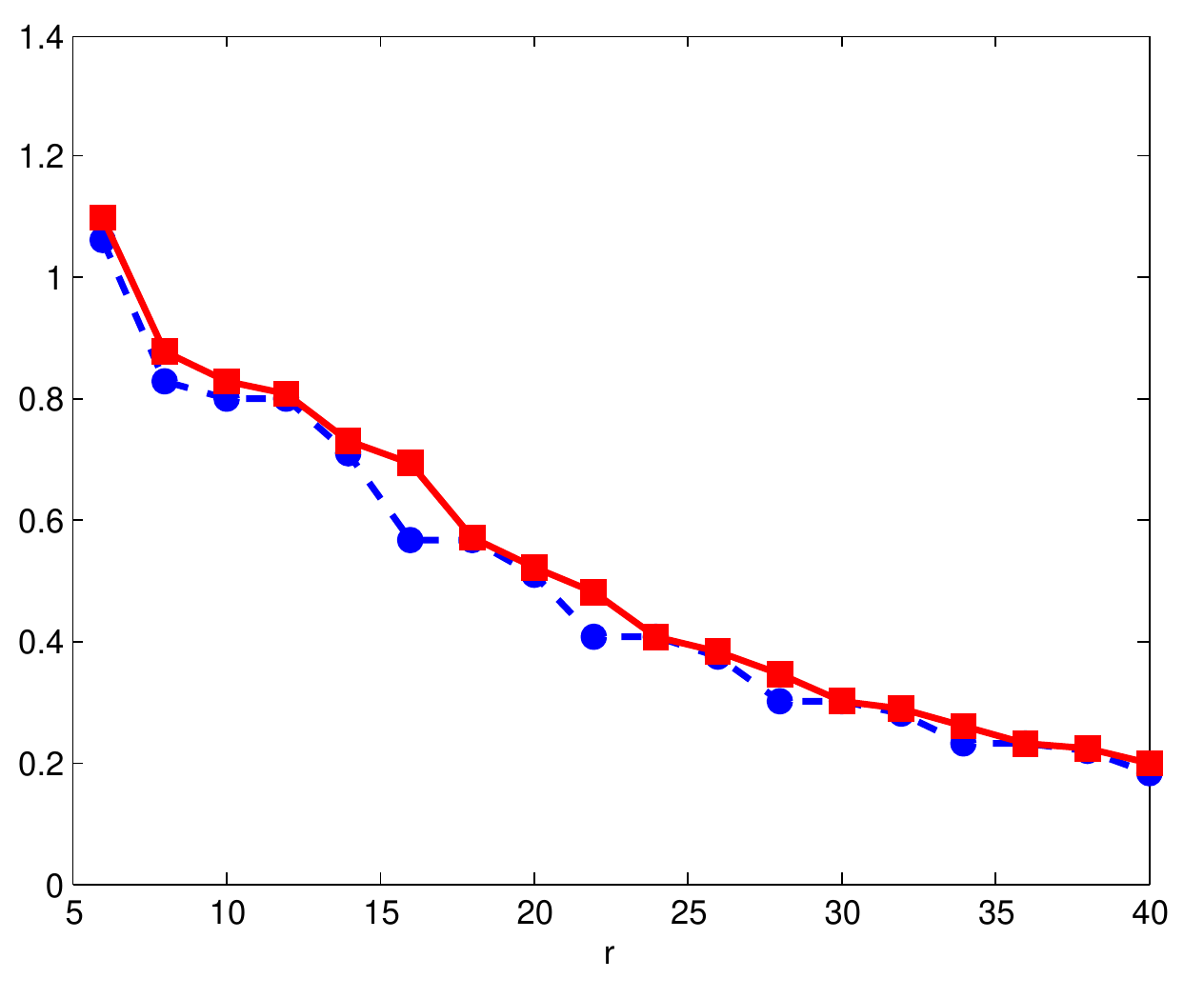}
	\caption{Motzkin polynomial}
\end{subfigure}
\vspace{.4cm}
  \caption{Graphical representation of Table \ref{table:result1} to illustrate the comparison of the upper bounds $\underline{f}_{\mathbf{K}}^{(r)}$ and $f^{(r)}$. The values $\underline{f}_{\mathbf{K}}^{(r)}$ are marked with circles connected by a dashed line and $f^{(r)}$ with squares connected by a solid line.} \label{fig:result1}
\end{center}
\end{figure}
 
 Finally, it is shown in \cite{ConvAna} that one may obtain  feasible points corresponding to bounds like $f^{(r)}$ through sampling from the probability distribution defined by the optimal density function. In particular, one may use the  {\em method of conditional distributions}  (see e.g., \cite[Section 8.5.1]{law07}).
 For $\mathbf{K} = [0,1]^n$, the procedure is described in detail in \cite[Section 3]{ConvAna}.

\section*{Appendix}

\subsection*{A. Proof of Lemma \ref{lemChebyshev}}

We give here a proof of Lemma \ref{lemChebyshev}, which we repeat for convenience.

\noindent\textbf{Lemma \ref{lemChebyshev}}
\textit{For any fixed integer $k > 1$,
one has
\begin{align*}
\max_{0\le i \le k-1} |u_i^{(k)}| \le \max_{0\le i \le k} |t_i^{(k)}| = 2^{k-1-2\psi(k)}\frac{k(k-\psi(k)-1)!}{\psi(k)!(k-2\psi(k))!},\tag*{\eqref{Cheb coeff bound}}
\end{align*}
where  $\psi(k) = 0$ for $k\leq 4$ and $\psi(k)=\left\lceil\frac{1}{8}\left(4k-5 - \sqrt{8k^2-7}\right)\right\rceil$ for $k\geq 4$. Moreover, the right-hand side of the equation increases monotonically with increasing $k$.}

\begin{proof}
We recall the representation of the Chebyshev polynomials in the monomial basis:
\begin{align*}
	T_k(x) &= \sum_{i=0}^k t_i^{(k)}x^i = \frac{k}{2}\sum_{m=0}^{\lfloor \frac{k}{2}\rfloor} (-1)^m\frac{(k-m-1)!}{m!(k-2m)!}(2x)^{k-2m}, &k>0,\\
	U_{k-1}(x) &= \sum_{i=0}^{k-1}u_i^{(k)}x^i = \sum_{m=0}^{\lfloor \frac{k-1}{2}\rfloor} (-1)^m\frac{(k-m-1)!}{m!(k-1-2m)!}(2x)^{k-1-2m}, &k>1.
\end{align*}
So, concretely, the coefficients are given by
\begin{align*}
	t_{k-2m}^{(k)} &= (-1)^m\cdot 2^{k-1-2m}\cdot \frac{k(k-m-1)!}{m!(k-2m)!},& k>0,\ 0\leq m \leq \left\lfloor \frac{k}{2}\right\rfloor,\\
	u_{k-1-2m}^{(k)} &= (-1)^m\cdot 2^{k-1-2m}\cdot \frac{(k-m-1)!}{m!(k-1-2m)!}, & k>1,\ 0 \leq m\leq \left\lfloor \frac{k-1}{2}\right\rfloor.
\end{align*}
It follows directly that $t_{k-2m}^{(k)} = \frac{k}{k-2m} u_{k-1-2m}^{(k)} $ and thus $|t_{k-2m}^{(k)}|> |u_{k-1-2m}^{(k)}|$ for $m<\frac{k}{2}$ and all $k>1$ which implies the inequality on the left-hand side of \eqref{Cheb coeff bound}.

Now we show that the value of $\max_{0\le m\le \left\lfloor \frac{k}{2}\right\rfloor} |t_{k-2m}^{(k)}|$ is attained for $m=\psi(k)$. 
 For this we examine the quotient
\begin{equation}\label{quotient}\tag{A.1}
	\frac{|t_{k-2(m+1)}^{(k)}|}{|t_{k-2m}^{(k)}|} 
		= \frac{(k-2m)(k-2m-1)}{4(m+1)(k-m-1)}\\
		= \frac{k^2-4mk+4m^2+2m-k}{4mk-4m^2-8m+4k-4}.
\end{equation}
Observe that this quotient is at most 1 if and only if  $m_1\le m \le m_2$, where we set
$m_1=\frac{1}{8}\left(4k-5 - \sqrt{8k^2-7}\right)$ and $m_{2} =\frac{1}{8}\left(4k-5 + \sqrt{8k^2-7}\right)$.
Hence the function  $m \mapsto |t_{k-2m}^{(k)}|$ is monotone increasing for $m \le  m_1$ and
monotone decreasing for $m_1 \le m\le m_2$.
Moreover, as $\lfloor m_1\rfloor \le m_1$, we deduce that $|t^{(k)}_{k-2\lceil m_1\rceil}|\ge |t^{(k)}_{k-2\lfloor m_1\rfloor }|$.
Observe furthermore that  $m_1\ge 0$ if and only if $k\ge 4$, and $m_2\ge \frac{k}{ 2}$ for all $k>1$.

Therefore, in the case $k\ge 4$, $\max_{0\le m\le \left\lfloor \frac{k}{2}\right\rfloor} |t_{k-2m}^{(k)}|$ is attained at  $\lceil m_1\rceil =  \psi(k),$
and thus it is equal to $|t^{(k)}_{k-2\psi(k)}|$.
In the case $1<k\leq4$,
$\max_{0\le m\le \left\lfloor \frac{k}{2}\right\rfloor} |t_{k-2m}^{(k)}|$ is attained at $m=0$, and thus it is equal to $|t^{(k)}_k|= 2^{k-1}.$



Finally we show  that the rightmost term of \eqref{Cheb coeff bound} increases monotonically with $k$. We show the inequality:
$|t^{(k)}_{k-2\psi(k)}| \le |t^{(k+1)}_{k+1-2\psi(k+1)}|$ for $k\ge 4$.
For this we consider again the sequence of Chebyshev coefficients, but this time we are interested in the behavior for increasing $k$, i.e., in the map $k \mapsto |t_{k-2m}^{(k)}|$. So, for fixed $m$, we consider the quotient
\begin{equation*}
	\frac{|t_{k+1-2m}^{(k+1)}|}{|t_{k-2m}^{(k)}|} = \frac{2^{k-2m}(k+1)(k-m)!\, m!\, (k-2m)!}{2^{k-1-2m}k(k-m-1)!\, m!\, (k+1-2m)!} = 2\cdot\frac{k+1}{k}\cdot\frac{k-m}{k+1-2m},
\end{equation*}
which is equal to 2 if $m=0$, and at least 1 if $m>0$  since every factor is at least 1. Thus, for $m=\psi(k)$,  we obtain
\begin{equation}\label{eqmono1}\tag{A.2}
|t^{(k)}_{k-2\psi(k)}| \le |t^{(k+1)}_{k+1-2\psi(k)}|. 
\end{equation}
Consider the map $\phi\colon [4,\infty) \to\mathbb{R},\ k\mapsto\phi(k)=
 \frac{1}{8}\left(4k-5 - \sqrt{8k^2-7}\right)$, so that $\psi(k)=\lceil \phi(k)\rceil$.
The map $\phi$  is monotone increasing, since its derivative $\phi'(k) = \frac{1}{8}\left(4-\frac{16k}{2\sqrt{8k^2-7}}\right) = \frac{\sqrt{8k^2-7}-2k}{2\sqrt{8k^2-7}}$ is positive for all $k\geq 4$. Hence, we have: $\psi(k)\le \psi(k+1)$.
Then, in view of \eqref{quotient} (and the comment thereafter), we have
$   |t^{(k+1)}_{k+1-2m}| \le |t^{(k+1)}_{k+1-2(m+1)}| \ \text{ if } m\le \psi(k+1),$ and thus
\begin{equation}\label{eqmono2}\tag{A.3}
  |t^{(k+1)}_{k+1-2\psi(k)}| \le |t^{(k+1)}_{k+1-2\psi(k+1)}|.
  \end{equation}
Combining (\ref{eqmono1}) and (\ref{eqmono2}), we obtain the desired inequality:
$|t^{(k)}_{k-2\psi(k)}| \le |t^{(k+1)}_{k+1-2\psi(k+1)}|.$
\qed\end{proof}

\subsection*{B. Useful identities for the Chebychev polynomials}

Recall the notation $d\mu(x)$ to denote the Lebesgue measure with the function
$ \prod_{i=1}^n\left(\pi \sqrt{1-x_i^2}\right)^{-1}$ as density function.
In order to compute the matrices $A^I$ and $B^I$ we need to evaluate the following integrals:
$$\langle T^\alpha, T^\beta T^\gamma \prod_{i\in I} (1-x_i^2)\rangle =
\prod_{i\in I} \int_{-1}^1 T_{\alpha_i}(x_i)T_{\beta_i}(x_i) T_{\gamma_i}(x_i) (1-x_i^2)d\mu(x_i) \cdot \prod_{i\not\in I}  \int_{-1}^1 T_{\alpha_i}(x_i)T_{\beta_i}(x_i) T_{\gamma_i}(x_i) d\mu(x_i). $$
Thus we can now assume that we are in the univariate case. Suppose  we are given integers $a,b,c\ge 0$ and the goal is to evaluate the integrals
\[
\int_{-1}^1 T_a(x) T_b(x) T_c(x)d\mu(x) \ \text{ and } \ \int_{-1}^1 T_a(x) T_b(x) T_c(x)(1-x^2)d\mu(x).
\]
  We use the following identities for the (univariate) Chebyshev polynomials:
  \[
T_aT_b=\tfrac{1}{ 2}(T_{a+b}+T_{|a-b|}),\ T_aT_bT_c =\tfrac{1}{ 4}( T_{a+b+c}+T_{|a+b-c|} +T_{|a-b|+c}+T_{||a-b|-c|}),
\]
so that
\begin{eqnarray*}
T_aT_bT_cT_2&=& \tfrac{1}{ 8}(T_{a+b+c+2} + T_{|a+b+c-2|} +T_{|a+b-c|+2} +T_{||a+b-c|-2|}  \\
&&+ T_{|a-b|+c+2}+T_{||a-b|+c-2|}+ T_{||a-b|-c|+2}+T_{|||a-b|-c|-2|}).
\end{eqnarray*}
Using the orthogonality relation $\int_{-1}^1 T_a d\mu(x) = \delta_{0,a}$, we obtain that
\[
\int_{-1}^1 T_aT_bT_cd\mu(x)= \tfrac{1}{ 4} (\delta_{0,a+b+c} +\delta_{0, a+b-c} +\delta_{0, |a-b|+c} +\delta_{0,|a-b|-c}).
\]
Moreover,
using the fact that $1-x^2=(1-T_2)/2$,
we get
\[
\int_{-1}^1 T_aT_bT_c(1-x^2)d\mu(x) = \frac{1}{ 2} \int_{-1}^1 T_aT_bT_c(1-T_2)d\mu(x)
=\frac{1}{ 2} \int_{-1}^1 T_aT_bT_c d\mu(x) - \frac{1}{ 2} \int_{-1}^1 T_aT_bT_cT_2d\mu(x),
\]
and thus
\begin{eqnarray*}
\int_{-1}^1 T_aT_bT_c(1-x^2)d\mu(x) &=& \tfrac{1}{ 8}(\delta_{0,a+b+c} +\delta_{0, a+b-c} +\delta_{0, |a-b|+c} +\delta_{0,|a-b|-c}) \\
&&
- \tfrac{1}{ 16}(\delta_{0,a+b+c-2}+\delta_{0,|a+b-c|-2} +\delta_{0, |a-b|+c-2}
+\delta_{0, ||a-b|-c|-2} ).
\end{eqnarray*}

\subsection*{C. Test functions}
\begin{description}
	\item[Booth function] $n=2$,
		$f_{\min} = f(0.1,0.3) = 0$, $f([-1,1]^2) \approx [0,2\, 500]$
		\begin{multline*}		
				f(x) = (10x_1+20x_2-7)^2 + (20x_1+10x_2-5)^2\\
				= 250(T_2(x_1)+T_2(x_2))+800\, T_1(x_1)T_1(x_2)-340\, T_1(x_1)-380\, T_1(x_2)+574.
		\end{multline*}
	\item[Matyas function] $n=2$,
		$f_{\min} = f(0,0) = 0$, $f([-1,1]^2) \approx [0,100]$
		\begin{equation*}				
				f(x) = 26(x_1^2+x_2^2)-48x_1x_2
				= 13(T_2(x_1)+T_2(x_2)) - 48T_1(x_1)T_1(x_2) + 26.
		\end{equation*}
	\item[Motzkin polynomial] $n=2$,
		$f_{\min} = f(\pm \tfrac{1}{2},\pm \tfrac{1}{2}) = 0$, $f([-1,1]^2) \approx [0,80]$
		\begin{multline*}
				f(x) = 64(x_1^4x_2^2+x_1^2x_2^4) - 48x_1^2x_2^2 +1
				= 4(T_4(x_1)+T_4(x_1)T_2(x_2)\\+T_2(x_1)T_4(x_2)+T_4(x_2)) + 20\,T_2(x_1)T_2(x_2) + 16\,(T_2(x_1)+T_2(x_2)) + 13.
		\end{multline*}
	\item[Three-Hump Camel function] $n=2$,
		$f_{\min} = f(0,0)= 0$, $f([-1,1]^2) \approx [0,2\, 000]$
		\begin{multline*}				
				f(x) = \frac{5^6}{6}x_1^6-5^4\cdot 1.05x_1^4+50x_1^2+25x_1x_2+25x_2^2\\
				= \frac{5^6}{192}\, T_6(x_1) +\frac{1625}{4}\, T_4(x_1) + \tfrac{58725}{64}\, T_2(x_1) +25\, T_1(x_1)T_1(x_2) + 12.5\, T_2(x_2) + \tfrac{14525}{24}.
		\end{multline*}
	\item[Styblinski-Tang function] $n = 2,\, 3$, $f_{\min} = -39.17\cdot n$,$f([-1,1]^2 \approx [-70,200]$
		\begin{equation*}	
				f(x) = \sum_{j=1}^n 312.5x_j^4-200x_j^2+12.5x_j
				=  \sum_{j=1}^n \left(\frac{625}{16}\, T_4(x_j) + \frac{225}{4}\, T_2(x_j) + \frac{25}{2}\, T_1(x_j) + \frac{275}{16}\right).
		\end{equation*}
	\item[Rosenbrock function]
		$n = 2,\ 3$, $f_{\min} = 0$, $f([-1,1]^2) \approx [0,4\, 000]$
		\begin{multline*}
				f(x) = \sum_{j=1}^{n-1} 100(2.048\cdot x_{j+1}-2.048^2\cdot x_j^2)^2 + (2.048\cdot x_j-1)^2\\
				= \sum_{j=1}^{n-1} \left[ 12.5\cdot 2.048^4\, T_4(x_j) - 100\cdot 2.048^3\, T_2(x_j)T_1(x_{j+1})
				+ (0.5 + 50\cdot 2.048^2) 2.048^2\, T_2(x_j)\right.\\
				\left. + 50\cdot 2.048^2\, T_2(x_{j+1}) -4.096\, T_1(x_j) - 100\cdot 2.048^3 \, T_1(x_{j+1})+1 +2.048^2(37.5\cdot 2.048^2+50.5)\right].
		\end{multline*}
\end{description}

\end{document}